\documentclass{article}
\usepackage{latexsym, amsmath, amssymb}
\usepackage{graphicx}
\usepackage{soul}
\usepackage{color}
\usepackage{float}
\usepackage{amsmath}
\usepackage{epstopdf}
\usepackage{amsmath, amsthm, amscd, amsfonts}
\usepackage{amsmath}
\usepackage{array}
\usepackage{hyperref}
\usepackage{multirow}
\usepackage{amssymb}
\usepackage{amsfonts}
\usepackage{amsthm}
\usepackage{float}
\usepackage{xcolor}
\usepackage{rotating}
\usepackage{longtable} 
\usepackage{enumerate}
\usepackage{color,soul}
\usepackage[a4paper, total={7in, 10in}]{geometry}
\usepackage{lineno}
\newtheorem{exm}{Example}[section]
 \newtheorem{prop}{Proposition}[section]
 \newtheorem{defn}{Definition}[section]
\newtheorem{rem}{Remark}[section]
\newtheorem{theorem}{Theorem}[section]
 \newtheorem{lem}{Lemma}[section]
\makeatletter
\newcommand*{\rom}[1]{\expandafter\@slowromancap\romannumeral #1@}
\makeatother

\usepackage{xcolor}

\def\bege{\begin{equation}} \def\ende{\end{equation}}

   \def\begr{\begin{eqnarray}}
\def\endr{\end{eqnarray}} 
 
\def\bege{\begin{equation}} \def\ende{\end{equation}}
\def\begr{\begin{eqnarray}} \def\endr{\end{eqnarray}} \def\bnum{\begin{enumerate}} \def\enum{\end{enumerate}}

\begin{document}
\begin{center}{\Large
 \textbf{Fault-tolerant metric basis and dimension of barycentric subdivision of zero divisor graphs}}

 \begin{center}
Vidya S$^{1}$, Sunny Kumar Sharma$^{2}$ Prasanna Poojary $^{1}$$^{\ast}$, Vadiraja Bhatta G R$^{1}$, 
\end{center}

\begin{center}
 $^{1}$ Manipal Institute of Technology, Manipal Academy of Higher Education, Manipal, Karnataka, India\\
 $^{2}$ School of Mathematics, Shri Mata Vaishno Devi University, Katra-182320, Jammu and Kashmir, India\\ 
 % {\it \textcolor{blue}{\underline{vidyassdk@gmail.com}}}\\
 
 % {\it \textcolor{blue}{\underline{sunny.sharma@manipal.edu; sunnysrrm94@gmail.com}}}\\
 % $^{3}$ Department of Mathematics, Manipal Institute of Technology Bengaluru, \\Manipal Academy of Higher Education, Manipal, Karnataka, India\\ 
 % {\it \textcolor{blue}{\underline{poojary.prasanna@manipal.edu; poojaryprasanna34@gmail.com}}}\\
 % $^{3}$ Department of Mathematics, Manipal Institute of Technology, \\Manipal Academy of Higher Education, Manipal, Karnataka, India\\
 % {\it \textcolor{blue}{\underline{ vadiraja.bhatta@manipal.edu}}}\\
 
\end{center}
 \end{center}
 \textbf{Abstract} The undirected zero divisor graph of a commutative ring with unity \( R \), denoted by \( \Gamma(R) = (V(\Gamma(R)), E(\Gamma(R))) \). The vertex set \( V(\Gamma(R)) \) consists of all the non-zero zero-divisors of \( R \). The edge set \( E(\Gamma(R)) \) is defined by the set \( \{ e = a_1 a_2 \mid a_1 \cdot a_2 = 0 \text{ and } a_1, a_2 \in V(\Gamma(R)) \} \). The barycentric subdivision of $\Gamma$ is the process of subdividing each edge by inserting new vertex in the graph $\Gamma$. In this article, we have focused on the fault-tolerant metric dimension of the barycentric subdivision of zero divisor graph of the group of integers modulo \( n \), represented by \( fdim(BS(\Gamma(\mathbb{Z}_n )\), where \( n = pq \); \( p \) and \( q \) are distinct odd primes with \( q > p \). We also demonstrate that \( fdim(BS(\Gamma(\mathbb{Z}_n) \geq q - 1 \) for every \( n = pq \), where \( p \) and \( q \) are any distinct odd primes with \( q > p \). 
 $^{\ast}$Corresponding: \textcolor{blue}{\underline{poojary.prasanna@manipal.edu}}

% Let $R$ be a commutative ring with unity 1 and let $ G(V,E)$ be a simple, connected, nontrivial graph. Let $d(p,q)$ be the distance between the
% % vertices $p$ and $q$ in $G$. An undirected zero divisor graph of a ring $R$ is a graph, denoted by $\Gamma(R)=(V(\Gamma(R)),E(\Gamma(R)))$, 
% %  whose vertex set $V(\Gamma(R))$ consists of all the nonzero divisor of $R$ and 
% %  $E(\Gamma(R))=
% %   \{e=p_1p_2|p_1.p_2=0 
% %   \  p_1, p_2 \in V(\Gamma(R))\}$. 
% %   In this article, we have considered a zero divisor graph of a group of integers modulo $n$, i.e, $\Gamma(\mathbb Z_n)$, where n=pq and p and q are two distinct primes with $q\geq 2p-1$. We determine its metric  and FTMD of the barycentric subdivision of the zero divisor graph $\Gamma(\mathbb Z_n)$.\\\\

\section{Introduction} 
Let \( G(V, E) \) be a simple graph with a vertex set \( V(G) \) and an edge set \( E(G) \). The degree of a vertex \( v \), denoted as \( d(v) \), indicates the number of edges incident with vertex $v$. A bipartite graph is a graph in which the vertex set is divided into two disjoint sets, denoted by \( V_1 \) and \( V_2 \). Every edge \( e = ab \) in \( E(G) \) connects a vertex \( a \) from one set to a vertex \( b \) from the other set, meaning either \( a \in V_1 \) and \( b \in V_2 \), or \( a \in V_2 \) and \( b \in V_1 \). A complete bipartite graph is a bipartite graph in which every vertex in \( V_1 \) is adjacent to every vertex in \( V_2 \). This complete bipartite graph, with \(|V_1| = m\) and \(|V_2| = n\) vertices, is denoted by \( K_{m,n} \). For foundational concepts in graph theory, we recommend the books by West \cite{w} and by Bondy and Murty \cite{bo}.\\

\noindent Beck \cite{beck} was the first to introduce the concept of a zero-divisor graph, denoted by \(\Gamma(R)\), for a commutative ring \(R\). His work established a connection between ring theory and graph theory, with a primary focus on finitely colorable rings. Subsequently, Anderson and Livingston \cite{anderson} presented a new perspective on associating zero-divisors with \(R\). They defined the undirected zero-divisor graph of a ring \(R\) as a graph \(\Gamma(R) = (V(\Gamma(R)), E(\Gamma(R)))\), where the vertex set \(V(\Gamma(R))\) comprises of all the non-zero zero-divisors of \(R\). The edge set \(E(\Gamma(R))\) is defined as follows: 

\[ 
E(\Gamma(R)) = \{e = p_1p_2 \,|\, p_1 \cdot p_2 = 0 \text{ and } p_1, p_2 \in V(\Gamma(R))\}. \]
We will follow the approach outlined by Anderson and Livingston throughout this paper, focusing exclusively on non-zero zero-divisors as the vertices of the graph \(\Gamma(R)\). Redmond \cite{redmond} later expanded the definition of a zero-divisor graph to include non-commutative rings. Furthermore, Demeyer et al. \cite{demyer} generalised the concept of zero-divisor graphs from rings to semigroups. Additionally, Lucas \cite{lucas} analysed the diameter of \(\Gamma(R)\). Behboodi \cite{behboodi} introduced the concept of zero-divisor graphs for modules over commutative rings. For more information on this topic, one can refer to \cite{ak, pir}.\\

\noindent The concept of metric dimension (MD) was introduced independently by Slater \cite{sl} and by Harary and Melter \cite{ha}. Slater focused on MD of trees, while Harary and Melter explored the MD of complete graphs, cycles, and bipartite graphs. Slater referred the resolving set and metric dimension as locating set and location number, respectively. In contrast, Harary and Melter used the terms resolving set and MD. This paper will consistently use the terms resolving set and MD instead of locating set and location number. Chartrand and Zhang \cite{cz} made a significant advancement by proposing the use of members of the metric basis as sensors in various security applications. These sensors play a crucial role in identifying threats, but their effectiveness can be compromised if one of them malfunctions. Such failures could leave critical vulnerabilities in the system, making it difficult to adequately respond to potential intruders, whether they are threats like fire or burglary. To counteract this challenge, Hernando et al. \cite{her} introduced the concept of fault-tolerant metric dimension (FTMD). This innovative approach is designed to enhance the reliability of the sensor network. The central idea of a fault-tolerant resolving set (FTRS) is to ensure that the network remains functional in the event that one sensor fails. With this design, the remaining sensors can still collaboratively identify and respond to intrusions, thereby reducing the risk of system failure. Ultimately, this development underscores that the potential applications of FTMD extend across a wide range of fields. This versatility makes FTMD an essential consideration in the design of robust and reliable security systems.\\

\noindent Determining the FTMD of a graph presents an interesting and complex combinatorial challenge, with significant potential applications in sensor networks. This area of research has so far examined a variety of basic families of graphs, leading to important findings. In their foundational article, Hernando et al. \cite{her} not only established the FTMD specifically for trees but also provided a theoretical upper bound applicable to the FTMD of arbitrary graphs. This work set the stage for further exploration into more complex structures. Expanding upon this foundation, Raza et al. \cite{hara} focused their research on convex polytopes, determining their FTMD. They highlighted a particularly significant result: certain infinite families of convex polytopes maintain a constant FTMD, which has profound implications for the design and analysis of various network structures. Following that, Saha et al. \cite{saha} investigated the FTMD of circulant graphs. Recently, Sharma and Bhat \cite{sbha} focused on the FTMD heptagonal circular ladder. Their work emphasizes the growing interest in specific graph structures and their unique properties concerning fault tolerance. Those interested in exploring further applications of the FTMD can refer to \cite{salman}.
 \\

\noindent Pirzada and Raja \cite{pir} were the first to investigate the MD of a zero-divisor graph. They examined the MD of $\Gamma(R \times F_s)$, where $F_s$ is a finite field. Additionally, they computed the MD of $\Gamma(R_1 \times R_2 \times R_3 \ldots \times R_r)$, where all $R_i$ (for $1 \leq i \leq r$) are finite commutative rings. Subsequently, Pirzada et al. \cite{pi} determined both the MD and upper dimensions of zero-divisor graphs associated with commutative rings, as well as various properties of the MD in these graphs. Later, Sharma and Bhat \cite{ss} investigated the FTMD of a zero-divisor graph and also examined the line graph of a zero-divisor graph associated with a commutative ring. \\

\noindent  An operation that involves splitting an edge into two edges by inserting a new vertex into the interior of the edge is called subdivision an edge. When this operation is applied to a sequence of edges in a graph $G$, it is referred to as subdividing the graph \( G \). The resulting graph is known as the subdivision of graph \( G \). If all edges in a graph are subdivided, it is called a barycentric subdivision of the graph (\cite{koam}). Our research will investigate the FTMD of the barycentric subdivision of a certain type of zero divisor graph.\\

 \noindent  In the next section, we will discuss the basic definitions and some results related to MD, FTMD, and zero divisor graphs. Following that, in Section 3, we will examine the FTMD of a barycentric subdivision of \(\Gamma(Z_{pq})\) for cases where \(p \geq 3\), and \(q \geq 2p - 3\). Finally, we will conclude that the FTMD of the barycentric subdivision of \(\Gamma(Z_{pq})\) is greater than or equal to \(q - 1\), where \(p\) and \(q\) are distinct odd primes with \(p \geq 3\) and \(q > p\).
\section{Preliminaries}
This section includes the fundamental definitions and results on MD, FTMD, and zero divisor graphs, which are essential for proving our main results.
\begin{defn}\cite{su}
\noindent The neighbourhood (open neighbourhood) of a vertex $p$ in a graph $G=(V, E)$, is denoted by $N(p)$, and is defined as 
$N(p) = \{x \in V(G) | px \in E(G)\}.$ The closed neighborhood of a vertex $p$, denoted by $N[p]$, it is the set $\{p\} \cup
N(p)$.
 \end{defn}
  \begin{defn}\cite{cht} 
A vertex $v\in V(G)$ resolves a pair of vertices $x_1$ and $x_2$, if $d(v,x_1)\neq d(v,x_2)$, where $x_1, x_2 \in V(G)$.
 \end{defn}
 \begin{defn}\cite{ha}		
\noindent Let  $A=\{v_1, v_2, v_3, \ldots,v_k\}$ be an ordered set of vertices in $G$. The metric coordinate$/$metric code of a vertex $s$ is denoted by $\delta(s|A)$, and defined as $\delta(s|A)=(d(s,v_1),d(s,v_2),
d(s,v_3),...,d(s,v_k))$. The set $A$ is a resolving set if every two different vertices have a distinct representation with respect to the set $A$.
   \end{defn}   
    \begin{defn}\cite{ha} 
  A resolving set with minimum cardinality is called the metric basis for $G$ and its cardinality is termed as the MD of $G$. It is usually denoted by $dim(G).$ 
 \end{defn}  
 \begin{defn} \cite{her}
    A resolving set $S$ of a graph $G$ is considered as a FTRS if removing any vertex from $S$ still results in a resolving set. The cardinality of the minimum FTRS is known as the FTMD, denoted by $fdim(G)$.
 \end{defn}
\noindent\begin{prop} \cite{her}
A graph $G=(V, E)$ has FTMD two $\iff$ $G$ is a path.
\end{prop}
\begin{prop} \cite{77}\label{ET}
Let $p$ and $q$ be two distinct primes, where $q > p$. Suppose $R=\mathbb{Z}_n$, where $n=pq$, then the following holds:
 \begin{enumerate}
    \item  dim($BS(\Gamma(R))
)=q-2$ if $q\geq 2p-1.$
 \item  dim($BS(\Gamma(R)) > q-2$ if  $p+1 < q < 2p-3.$
  \item  dim($BS(\Gamma(R))=q-1$, if  $q = 2p-3.$
\end{enumerate} 
 \end{prop}
 \noindent According to the definition of the FTMD of a graph \( G \), the following inequality is implied:
\begin{equation}\label{a}
fdim(\Gamma(R)\geq dim(\Gamma(R))+1.
\end{equation}
This inequality indicates that the FTMD is always at least one unit greater than the MD, reflecting the added complexity of maintaining resolvability in the face of potential vertex failures. This relationship highlights the enhanced requirements for a basis set that can withstand the removal of vertices while still effectively allowing for the identification of all other vertices in the graph.
\section{ Fault-Tolerant Metric Dimension of Barycentric Subdivision of \texorpdfstring{$\Gamma (\mathbb{Z}_n$)}{Gamma(Zn)}}
In this section, we will determine the FTMD of $BS(\Gamma(\mathbb{Z}_n))$, where $n = pq$, and $p$ and $q$ are distinct odd primes with $q>p$. \\

% \begin{rem} \label{last} \quad \\
%     \begin{itemize}
%         \item As shown in Figure \ref{fig:4}, for \( p \geq 3 \), the set \( U \) contains \( p-1 \) vertices. Among these vertices, exactly \( \frac{p-1}{2} \) are adjacent only to vertices in the set \( S \) and are not adjacent to any vertices in other sets. Similarly, exactly \( \frac{p-1}{2} \) vertices are adjacent only to vertices in the set \( T \) and are not adjacent to any vertices in other sets. 
%     \item From the Figure. \ref{fig:4}, it is evident that if $d(u_i,a)\neq 1$, and  $d(u_j,a)\neq 1$, where $1\leq i, j \leq p-1$, $i\neq j$, and $a\in V(BS(\Gamma(\mathbb{Z}_{pq})))\backslash\{u_i,u_j\}$, then $d(u_i,a)=d(u_j,a)$.\\     
% \end{itemize}
% \end{rem}
\begin{theorem}\label{b}
Let $p$ and $q$ be two distinct odd primes, where $q > 2p-1$. Suppose $R=\mathbb{Z}_n$, where $n=pq$, then
  $fdim(BS(\Gamma(R))
)=q-1$.
 \end{theorem}
 \begin{proof}
 % \textbf{Case 1: for $p>2$ $\&$ $q> 2p-1$\\
\noindent Consider a ring $R =\mathbb{Z}_n$, where $n=pq$, and let $\Gamma(R)$ represent its zero divisor graph. We define $BS(\Gamma(R))$ as the barycentric subdivision of $\Gamma(R)$. The graph $BS(\Gamma(R))$ consists of $pq-1$ vertices and $2(p-1)(q-1)$ edges. In particular, We have partitioned the set $V(BS(\Gamma(R))$ into $4$ sets namely, $A=\{r_1,r_2,\ldots,r_{q-1}\}$, $S=\{s^\tau_i\}$, $T=\{t^\tau_i\}$, where $1\leq \tau\leq \frac{p-1}{2}$, and $1\leq i\leq {q-1}$ ($S^\tau=\{s^\tau_1,s^\tau_2,\ldots,s^\tau_{q-1}\}$, where $1\leq \tau \leq \frac{p-1}{2}$ are the subsets of $S$, and $T^\tau=\{t^\tau_1,t^\tau_2,\ldots,t^\tau_{q-1}\}$, where $1\leq \tau \leq \frac{p-1}{2}$  are the subsets of $T$),  and $U=\{u_1,u_2, \ldots, u_{p-1}\}$. These sets are illustrated in the Figure. \ref{fig:4}.
\begin{figure}[ht!]
    \includegraphics[width=14cm]{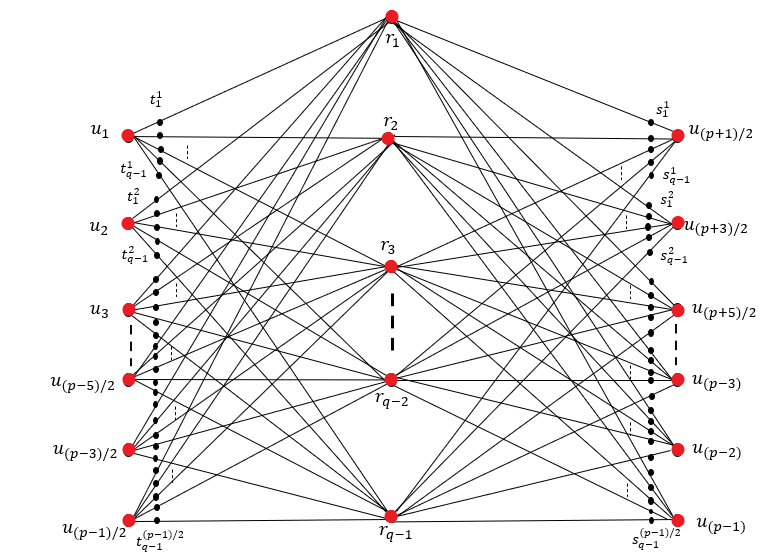}
    \caption{Barycentric Subdivision of Zero Divisor Graph $\mathbb{Z}_{pq}$}
    \label{fig:4}
\end{figure}
 
\noindent From Proposition \ref{ET}, we know that the MD of the barycentric subdivision of $\Gamma (R)$ is $q-2$, when $p>2$ $\&$ $q> 2p-1$. We now need to prove that its FTMD is $q-1$. We must demonstrate that $E$ is a minimum FTRS consisting of $q-1$ vertices, where $E=\{ r_1, t^1_2,t^1_3,t^2_4,t^2_5, \ldots,t^\frac{p-1}{2}_{p-1},t^\frac{p-1}{2}_{p}, s^1_{p+1}, s^1_{p+2}, \ldots, s^\frac{p-5}{2}_{2p-6}, s^\frac{p-5}{2}_{2p-5},s^\frac{p-3}{2}_{2p-4},
s^\frac{p-3}{2}_{2p-3}, s^\frac{p-1}{2}_{2p-2},\ldots, s^\frac{p-1}{2}_{q-1} \}$. \\

 The metric coordinates for the vertices $\{r_ \gamma:2\leq \noindent \gamma\leq q-1\}$ are given below
 \begin{equation*}
               \delta(r_{\gamma}|E)=
                   \begin{cases}
                   (4_{\{1\}},1_{\{ 2 \}},3_{\{3\}},\ldots,3_{\{q-3\}},3_{\{q-2\}}, 3_{\{q-3\}},3_{\{q-1\}}) &\textit{if}\quad \gamma= 2;\\
                  
                    (4_{\{1\}},3_{\{2\}},\ldots,3_{\{\gamma -1\}},1_{\{ \gamma \}},3_{\{\gamma +1\}},\ldots,3_{\{q-2\}},3_{\{q-1\}}) &\textit{if}\quad 3\leq \gamma $ $\leq q-3  ;\\
                    (4_{\{1\}},3_{\{2\}},\ldots,3_{\{q-3\}},1_{\{q-2\}}, 3_{\{q-1\}}) &\textit{if}\quad \gamma= q-2;\\
                  
                    (4_{\{1\}},3_{\{2\}},3_{\{3\}},\ldots,3_{\{q-3\}},3_{\{q-2\}}, 1_{\{q-1\}} ) &\textit{if}\quad \gamma= q-1.\\
                \end{cases}
                     \end{equation*}
\noindent The metric coordinates for the vertices 
                     $\{t^1_\gamma: 1\leq \gamma\leq q-1 \}$ are given below
                        \begin{equation*}
                        \delta(t^{1}_{\gamma}|E)=
                   \begin{cases}
                   (1_{\{1\}},2_{\{2\}},2_{\{3\}},4_{\{4\}},\ldots,4_{\{q-2\}}, 4_{\{q-1\}})&\textit{if}\quad\gamma= 1;\\
                     (3_{\{1\}},2_{\{2\}}, 2_{\{3 \}},2_{\{ 4 \}}, 4_{\{ 5 \}},\ldots,4_{\{q-2\}}, 4_{\{q-1\}})&\textit{if}\quad  \gamma= 4;\\
                     (3_{\{1\}},2_{\{2\}}, 2_{\{3 \}},4_{\{ 4 \}}, \ldots, 4_{\{\gamma-1 \}},2_{\{\gamma \}}, 4_{\{\gamma+1 \}},\ldots,4_{\{q-2\}}, 4_{\{q-1\}})&\textit{if}\quad 5 \leq  \gamma  \leq q-3  ;\\
                     (3_{\{1\}},2_{\{2\}}, 2_{\{3 \}},4_{\{ 4 \}},\ldots, 4_{\{q-3\}},2_{\{q-2\}}, 4_{\{q-1\}})&\textit{if}\quad \gamma= q-2;\\
                    (3_{\{1\}},2_{\{2\}},,2_{\{3\}}, 4_{\{4\}},\ldots,4_{\{q-2\}}, 2_{\{q-1\}}) &\textit{if}\quad\gamma= q-1.\\\\
                    
                    \end{cases}
                    \end{equation*}

                    \noindent The metric coordinates for the vertices 
                     $\{t^2_\gamma: 1\leq \gamma\leq q-1 \}$ are given below
                        \begin{equation*}
                        \delta(t^{2}_{\gamma}|E)=
                   \begin{cases}
                   (1_{\{1\}},4_{\{2\}},4_{\{3\}},2_{\{4\}}, 2_{\{5\}}, 4_{\{6\}}, \ldots,4_{\{q-2\}}, 4_{\{q-1\}})&\textit{if}\quad\gamma= 1;\\
                   (3_{\{1\}},2_{\{2\}},4_{\{3\}},2_{\{4\}}, 2_{\{5\}}, 4_{\{6\}},\ldots,4_{\{q-2\}}, 4_{\{q-1\}})&\textit{if}\quad\gamma= 2;\\
                     (3_{\{1\}},4_{\{2\}},2_{\{3\}},2_{\{4\}}, 2_{\{5\}}, 4_{\{6\}},\ldots,4_{\{q-2\}}, 4_{\{q-1\}})&\textit{if}\quad\gamma= 3;\\
                     (3_{\{1\}},4_{\{2\}},4_{\{3\}},2_{\{4\}}, 2_{\{5\}}, 2_{\{6\}}, 4_{\{7\}}, \ldots,4_{\{q-2\}}, 4_{\{q-1\}})&\textit{if}\quad\gamma= 6;\\
                     (3_{\{1\}},4_{\{2\}}, 4_{\{3 \}},2_{\{ 4 \}}, 2_{\{ 5 \}}, 4_{\{ 6 \}}, \ldots, 4_{\{\gamma-1 \}},2_{\{\gamma \}}, 4_{\{\gamma+1 \}},\ldots,4_{\{q-2\}}, 4_{\{q-1\}})&\textit{if}\quad 7 \leq  \gamma  \leq q-3  ;\\
                      (3_{\{1\}},4_{\{2\}},4_{\{3\}},2_{\{4\}}, 2_{\{5\}}, 4_{\{6\}}, \ldots, 4_{\{q-3\}} ,2_{\{q-2\}}, 4_{\{q-1\}})&\textit{if}\quad\gamma= q-2;\\
                    (3_{\{1\}},4_{\{2\}},4_{\{3\}},2_{\{4\}}, 2_{\{5\}}, 4_{\{6\}},\ldots,4_{\{q-3\}} ,4_{\{q-2\}}, 2_{\{q-1\}})&\textit{if}\quad\gamma= q-1.\\
                    \end{cases}
                    \end{equation*}
                 \noindent The metric coordinates for the vertices 
                     $\{t^\tau_\gamma:3\leq \tau\leq \frac{p-1}{2}$  $\&$  $1\leq \gamma\leq q-1 \}$ are given below
                        \begin{equation*}
                        \delta(t^{\tau}_{\gamma}|E)=
                   \begin{cases}
                   (1_{\{1\}},4_{\{2\}},\ldots, 4_{\{2\tau-1\}},2_{\{2\tau\}},2_{\{2\tau+1\}},4_{\{2\tau+2\}},\ldots,4_{\{q-2\}}, 4_{\{q-1\}})&\\
                   \qquad \qquad \qquad\qquad\qquad\qquad\qquad\qquad\qquad\qquad\qquad\qquad\qquad\qquad\qquad \qquad \quad\textit{if}\quad\gamma= 1;\\
                     (3_{\{1\}},2_{\{ 2 \}},4_{\{3\}}, \ldots, 4_{\{2\tau-1\}},2_{\{2\tau\}},2_{\{2\tau+1\}},4_{\{2\tau+2\}},\ldots,4_{\{q-2\}}, 4_{\{q-1\}})&\\
                   \qquad \qquad \qquad\qquad\qquad\qquad\qquad\qquad\qquad\qquad\qquad\qquad\qquad\qquad\qquad \qquad \quad\textit{if}\quad \gamma= 2;\\
                   (3_{\{1\}},4_{\{2\}}, \ldots,4_{\{\gamma-1 \}},2_{\{ \gamma \}},4_{\{\gamma +1\}}, \ldots, 4_{\{2\tau-1\}},2_{\{2\tau\}},2_{\{2\tau+1\}},4_{\{2\tau+2\}},\ldots,4_{\{q-2\}}, 4_{\{q-1\}})&\\
                   \qquad \qquad \qquad\qquad\qquad\qquad\qquad\qquad\qquad\qquad\qquad\qquad\qquad\qquad\qquad \qquad \quad\textit{if}\quad  3 \leq  \gamma  \leq 2\tau-2  ;\\
                      (3_{\{1\}},4_{\{2\}}, \ldots,4_{\{2\tau-2 \}},2_{\{ 2\tau-1 \}},2_{\{2\tau\}},2_{\{2\tau+1\}},4_{\{2\tau+2\}},\ldots,4_{\{q-2\}}, 4_{\{q-1\}})&\\
                   \qquad \qquad \qquad\qquad\qquad\qquad\qquad\qquad\qquad\qquad\qquad\qquad\qquad\qquad\qquad \qquad \quad\textit{if}\quad\gamma= 2\tau-1 ;\\
                     (3_{\{1\}},4_{\{2\}}, \ldots,4_{\{2\tau*-+-
                     1\}},2_{\{2\tau\}},2_{\{2\tau+1\}},2_{\{ 2\tau+2 \}}, 4_{\{2\tau+3 \}},\ldots,4_{\{q-2\}}, 4_{\{q-1\}})&\\
                   \qquad \qquad \qquad\qquad\qquad\qquad\qquad\qquad\qquad\qquad\qquad\qquad\qquad\qquad\qquad \qquad \quad\textit{if}\quad  \gamma= 2\tau+2 ;\\
                    (3_{\{1\}},4_{\{2\}}, \ldots,4_{\{2\tau-1\}},2_{\{2\tau\}},2_{\{2\tau+1\}},4_{\{2\tau+2\}},\ldots,4_{\{\gamma-1 \}},2_{\{ \gamma \}},4_{\{\gamma +1\}},\ldots,4_{\{q-2\}}, 4_{\{q-1\}})&\\
                   \qquad \qquad \qquad\qquad\qquad\qquad\qquad\qquad\qquad\qquad\qquad\qquad\qquad\qquad\qquad \qquad \quad\textit{if}\quad  2\tau+3 \leq  \gamma  \leq q-3  ;\\
                     (3_{\{1\}},4_{\{2\}}, \ldots,4_{\{2\tau-1\}},2_{\{2\tau\}},2_{\{2\tau+1\}},4_{\{2\tau+2\}},\ldots,4_{\{q-3 \}},2_{\{ q-2 \}}, 4_{\{q-1\}})&\\
                   \qquad \qquad \qquad\qquad\qquad\qquad\qquad\qquad\qquad\qquad\qquad\qquad\qquad\qquad\qquad \qquad \quad\textit{if}\quad  \gamma=  q-2  ;\\
                    (3_{\{1\}},4_{\{2\}}, \ldots,4_{\{2\tau-1\}},2_{\{2\tau\}},2_{\{2\tau+1\}},4_{\{2\tau+2\}},\ldots,4_{\{q-2\}}, 2_{\{q-1\}}) &\\
                   \qquad \qquad \qquad\qquad\qquad\qquad\qquad\qquad\qquad\qquad\qquad\qquad\qquad\qquad\qquad \qquad \quad\textit{if}\quad\gamma= q-1.\\
                    \end{cases}
                    \end{equation*}
                    \noindent The metric coordinates for the vertices 
                     $\{u_\tau:1\leq \tau\leq {p-1} \}$ are given below
                        \begin{equation*}
                        \delta(u_{\tau}|E)=
                   \begin{cases}
                   (2_{\{1\}},1_{\{2\}}, 1_{\{3\}}, 3_{\{4\}},\ldots,3_{\{q-2\}}, 3_{\{q-1\}})&\textit{if}\quad  \tau=1 ;\\
                   (2_{\{1\}},3_{\{2\}}, \ldots, 3_{\{2\tau-1\}},1_{\{2\tau\}},1_{\{2\tau+1\}},3_{\{2\tau+2\}},\ldots,3_{\{q-2\}}, 3_{\{q-1\}})&\textit{if}\quad  2 \leq \tau  \leq p-2 ;\\
                     (2_{\{1\}},3_{\{2\}}, \ldots, 3_{\{2\tau-1\}},1_{\{2\tau\}},1_{\{2\tau+1\}},\ldots,1_{\{q-1\}})&\textit{if}\quad  \tau= p-1 .\\
                     
                     \end{cases}
                     \end{equation*}
\noindent The metric coordinates for the vertices 
                     $\{s^\tau_\gamma:1\leq \tau\leq \frac{p-3}{2} $ $ \& $  $1\leq \gamma\leq q-1 \}$ are given below
                        \begin{equation*}
                        \delta(s^{\tau}_{\gamma}|E)=
                   \begin{cases}
                   (1_{\{1\}},4_{\{2\}},\ldots,4_{\{2\tau+p-2\}},2_{\{p+2\tau-1\}},2_{\{p+2\tau\}},4_{\{p+2\tau+1\}},\ldots,4_{\{q-2\}}, 4_{\{q-1\}})&\\
                   \qquad \qquad \qquad\qquad\qquad\qquad\qquad\qquad\qquad\qquad\qquad\qquad\qquad\qquad\qquad \qquad \quad \textit{if}\quad\gamma= 1;\\
                    (3_{\{1\}},2_{\{2 \}},4_{\{3\}},\ldots,4_{\{2\tau+p-2\}},2_{\{p+2\tau-1\}},2_{\{p+2\tau\}},4_{\{p+2\tau+1\}},\ldots,4_{\{q-2\}}, 4_{\{q-1\}})&\\
                   \qquad \qquad \qquad\qquad\qquad\qquad\qquad\qquad\qquad\qquad\qquad\qquad\qquad\qquad\qquad \qquad \quad \textit{if}\quad   \gamma= 2;\\
                     
                   (3_{\{1\}},4_{\{2\}}, \ldots,4_{\{\gamma-1 \}},2_{\{ \gamma \}},4_{\{\gamma +1\}},\ldots,4_{\{2\tau+p-2\}},2_{\{p+2\tau-1\}},2_{\{p+2\tau\}},4_{\{p+2\tau+1\}},\ldots,4_{\{q-2\}}, 4_{\{q-1\}})&\\
                   \qquad \qquad \qquad\qquad\qquad\qquad\qquad\qquad\qquad\qquad\qquad\qquad\qquad\qquad\qquad \qquad \quad \textit{if} \quad  3 \leq  \gamma  \leq 2\tau+p-3 ;\\

                     (3_{\{1\}},4_{\{2\}}, \ldots,4_{\{\gamma-1 \}},2_{\{ 2\tau+p-2 \}},2_{\{p+2\tau-1\}},2_{\{p+2\tau\}},4_{\{p+2\tau+1\}},\ldots,4_{\{q-2\}}, 4_{\{q-1\}})&\\
                   \qquad \qquad \qquad\qquad\qquad\qquad\qquad\qquad\qquad\qquad\qquad\qquad\qquad\qquad\qquad \qquad \quad\textit{if}\quad \gamma = 2\tau+p-2 ;\\
                     (3_{\{1\}},4_{\{2\}}, \ldots,4_{\{2\tau+p-2\}},2_{\{p+2\tau-1\}},2_{\{p+2\tau\}},2_{\{p+2\tau+1\}}, 4_{\{p+2\tau+2\}}, \ldots,4_{\{q-2\}})&\\
                   \qquad \qquad \qquad\qquad\qquad\qquad\qquad\qquad\qquad\qquad\qquad\qquad\qquad\qquad\qquad \qquad \quad \textit{if}\quad\gamma= p+2\tau+1 ;\\
                    (3_{\{1\}},4_{\{2\}}, \ldots,4_{\{2\tau+p-2\}},2_{\{p+2\tau-1\}},2_{\{p+2\tau\}},4_{\{p+2\tau+1\}},\ldots,4_{\{\gamma-1 \}},2_{\{ \gamma \}},4_{\{\gamma +1\}},\ldots,4_{\{q-2\}}, 4_{\{q-1\}})&\\
                   \qquad \qquad \qquad\qquad\qquad\qquad\qquad\qquad\qquad\qquad\qquad\qquad\qquad\qquad\qquad \qquad \quad \textit{if} \quad p+2\tau+2 \leq  \gamma  \leq q-3  ;\\
                     (3_{\{1\}},4_{\{2\}}, \ldots,4_{\{2\tau+p-2\}},2_{\{p+2\tau-1\}},2_{\{p+2\tau\}},4_{\{p+2\tau+1\}},\ldots, 4_{\{q-3\}},2_{\{q-2\}}, 4_{\{q-1\}})&\\
                   \qquad \qquad \qquad\qquad\qquad\qquad\qquad\qquad\qquad\qquad\qquad\qquad\qquad\qquad\qquad \qquad \quad  \textit{if}\quad   \gamma = q-2  ;\\
                    (3_{\{1\}},4_{\{2\}}, \ldots,4_{\{2\tau+p-2\}},2_{\{p-1+2\tau\}},2_{\{p+2\tau\}},4_{\{p+2\tau+1\}},\ldots,4_{\{q-2\}}, 2_{\{q-1\}}) &\\
                   \qquad \qquad \qquad\qquad\qquad\qquad\qquad\qquad\qquad\qquad\qquad\qquad\qquad\qquad\qquad \qquad \quad \textit{if}\quad\gamma= q-1.\\
                      \end{cases}
                    \end{equation*}
                    \noindent The metric coordinates for the vertices 
                     $\{s^\tau_\gamma: \tau =\frac{p-1}{2} $ $\&$  $1\leq \gamma\leq q-1 \}$ are given below
                        \begin{equation*}
                        \delta(s^{\tau}_{\gamma}|E)=
                   \begin{cases}
                   (1_{\{1\}},4_{\{2\}},\ldots,4_{\{2p-3 \}},2_{\{2p-2 \}}, \ldots,2_{\{q-3\}},2_{\{q-2\}}, 2_{\{q-1\}})
                    &\textit{if}\quad\gamma= 1;\\
                    (3_{\{1\}},2_{\{2\}},4_{\{3\}}, \ldots, 4_{\{2p-3 \}},2_{\{2p-2 \}}, \ldots,2_{\{q-3\}},2_{\{q-2\}}, 2_{\{q-1\}}) 
                    &\textit{if}\quad\gamma= 2;\\ 
                   (3_{\{1\}},4_{\{2\}}, \ldots,4_{\{\gamma-1\}},2_{\{ \gamma\}},4_{\{\gamma +1\}}, \ldots, 4_{\{2p-3 \}},2_{\{2p-2 \}}, \ldots,2_{\{q-3\}},2_{\{q-2\}}, 2_{\{q-1\}})&\textit{if}\quad   3 \leq  \gamma  \leq 2p-4  ;\\
                      (3_{\{1\}},4_{\{2\}}, \ldots, 4_{\{2p-4 \}} 2_{\{2p-3 \}},2_{\{2p-2 \}}, \ldots,2_{\{q-3\}},2_{\{q-2\}}, 2_{\{q-1\}})&\textit{if}\quad   \gamma= 2p-3.\\
                   
                     \end{cases}
                    \end{equation*}

                    \noindent After analysing the above codes, it is evident that removing any vertex from the set \( E \) (i.e., \( E - \{a\} \), where \( a \in E \)) while ensuring that all vertices of $BS(\Gamma(R))$ have unique metric coordinates with respect to \( E - \{a\} \) preserves the property of the fault-tolerance in \( E \). This implies that \( E \) is indeed an FTRS, leading to the conclusion that \( fdim(\Gamma(R)) \leq q - 1 \). Furthermore, based on Proposition \ref{ET} and Equation \ref{a}, it is clear that the FTMD of $\ BS(\Gamma(R) )$ is at least \( q - 1 \). Therefore, we can conclude that \( E \) is a minimum FTRS, i.e., \( fdim(BS\Gamma(R)) = q - 1 \). This concludes the proof of the above theorem. \\

\end{proof}
\begin{rem}\label{22} \cite{77}
If $d(r_i,x)\neq 1$, and $d(r_j,x) \neq 1$, where $1\leq i, j \leq q-1$, $i\neq j$, and $x\in V(BS(\Gamma(\mathbb{Z}_{pq})))\backslash\{r_i,r_j\}$, then $d(r_i,x)=d(r_j,x)$.     
 \end{rem}
\begin{lem}\label{lemma1} Let $BS(\Gamma(\mathbb{Z}_{pq}))$ be a 
 graph, where $p$ and $q$ be two distinct odd primes, and $q=2p-1$. And the vertex set $V(BS(\Gamma(\mathbb{Z}_{pq}))$ be partitioned into the subsets
      $A=\{r_1,r_2,\ldots,r_{q-1}\}$,
      \( S = \bigcup \limits_{\tau=1}^{\frac{p-1}{2}} S^\tau \),
      \( T = \bigcup \limits_{\tau=1}^{\frac{p-1}{2}} T^\tau \), and
      $U=\{u_1,u_2, \ldots, u_{p-1}\}$.  
     Let \( W \)  be any metric basis of the graph \( BS(\Gamma(Z_{pq})) \). Then  
\begin{enumerate}
\item The set $W$ is independent and $W\cap U=\emptyset$.

    \item Each set $S^k$, $T^k$ (where $1\leq k \leq \frac{p-1}{2}$) contains exactly two vertices of $W$, except exactly one set.
    \item There exists $r_1~r_2\in A$, such that $N[r_1]\cap W=\emptyset$, $r_2\in W$, and $|N[r_i]\cap W|=1$, $\forall ~r_i\in A\setminus \{r_1\}$, where $[A\setminus\{ r_2\}]\cap W=\emptyset$
    \end{enumerate}
\end{lem}
 \begin{proof}
 \noindent Consider a ring $\mathbb{Z}_{pq}$, where $q = 2p-1$, and let $\Gamma(Z_{pq})$ represent its zero divisor graph. In particular. Let \( W \) be any metric basis of \( BS(\Gamma(Z_{pq})) \). The set \( W \) consists of \( 2p - 3 \) vertices, as stated in Proposition \ref{ET}. 
      
 \begin{description}
     \item[Case 1:] \label{case1} If any two vertices of $W$ are adjacent or at least one vertex of $W$ is chosen from $U$ then there exist at least two vertices of $A$, denoted as $r_{u} $ and $r_{v} $, where $1\leq u\neq v \leq q-1$, such that $N[r_u]\cap W=\emptyset=N[r_v]\cap W$. This is because \( |A| = q-1 \), $d(r_s,u_j)=2$ for all $r_s\in A$ and $u_j\in U$, and \( N[r_m] \cap N[r_s] = \emptyset \), where $1\leq m\neq s \leq q-1$. Hence, it follows from the Remark \ref{22} that $\delta(r_{u}|W) =\delta(r_{v}|W)$.\\
 \item[Case 2:] \label{case2}~
 \begin{enumerate}
 \item If at least two distinct sets, say \(S^{\nu_1}\) and \(S^{\nu_2}\) (or \(T^{\nu_1}\) and \(T^{\nu_2}\), or \(S^{\nu_1}\) and \(T^{\nu_3}\)), have exactly one vertex of \(W\), let these vertices be \(s^{\nu_1}_i\) and \(s^{\nu_2}_j\), where \(1 \leq i, j \leq q-1\), then $\delta( s^{\nu_2}_i|W) =\delta( s^{\nu_1}_{j}|W)$.
 \item If at least two distinct sets, say \(S^{\nu_1}\) and \(S^{\nu_2}\) (or \(T^{\nu_1}\) and \(T^{\nu_2}\), or \(S^{\nu_1}\) and \(T^{\nu_3}\)), does not have any vertex of \(W\), then $\delta( u_{\nu_1}|W) =\delta( u_{\nu_2}|W)$, where $u_{\nu_1}$, and $u_{\nu_2}$ is adjacent to vertices of \(S^{\nu_1}\) and \(S^{\nu_2}\) respectively. 
 \item If one set, say \( S^{\nu_1} \) (or \( T^{\nu_2} \)), contains a single vertex of \( W \), denoted by \( s^{\nu_1}_i \), while another set, say \( S^{\nu_3} \) (or \( T^{\nu_4} \))  does not contain any vertices of \( W \), where \( 1 \leq \nu_1, \nu_2, \nu_3, \nu_4 \leq \frac{p-1}{2} \), then it follows that \[ \delta(s^{\nu_1}_j | W) = \delta(s^{\nu_3}_i | W) \] 
In this, \( s^{\nu_1}_j \) is adjacent to a vertex of \( A \), denoted as \( r_u \), where $r_u$ is not adjacent to any vertex of \( W \).
\item If at least one set, say \( S^k \) (or \( T^k \)), where \( 1 \leq k \leq \frac{p-1}{2} \), contains more than one vertex of \( W \) then, there exists at least two distinct sets, say \(S^{\nu_1}\) and \(S^{\nu_2}\) (or \(T^{\nu_1}\) and \(T^{\nu_2}\), or \(S^{\nu_1}\) and \(T^{\nu_3}\)), where \(1 \leq \nu_1, \nu_2, \nu_3 \leq \frac{p-1}{2}\), containing at most one vertex of $W$, since $|W|=2p-3$. Hence, this leads to the same contradiction as mentioned above.

 \end{enumerate}
 Without loss of generality, assume that each set $S^
 k$ and $T^k$ contain 2 vertices of $W$, except $S^1$. Since $|W|=2p-3$, one vertex remains to be chosen. If this vertex is chosen from the set $S^1$, then \( \delta(q_{u}|W) = \delta(r_{v}|W) \). In this case, \( q_u \) is adjacent to the vertices of \( S^{1} \), while \( r_{v} \) is not adjacent to any vertices of \( W \). Since \( |A| = q - 1 \) and \( N[r_m] \cap N[r_s] = \emptyset \) for \( 1 \leq m \neq s \leq q-1 \), it follows that there exists at least one vertex of \( A \) that is not adjacent to any vertices of \( W \).
 \item[Case 3:] ~
 \begin{enumerate}
    \item If any vertex of $A$ which is adjacent to at least two vertices of $W$, then we get the same contradiction as mentioned in Case 1.
    \item From Case 2, we observe that exactly \(2p-2\) vertices of \(S \cup T\) belong to \(W\). Additionally, from Case 1, it is clear that no vertex of \(U\) is included in \(W\). This indicates that exactly one vertex of \(A\) is in \(W\). Given that \(|W| = 2p-3\) and \(|A| = 2p-2\), along with the fact that \(N[r_m] \cap N[r_s] = \emptyset\), we can conclude that there must be one vertex in \(A\) that is not in \(W\) and is not adjacent to any vertices in \(W\).
 \end{enumerate}

\end{description}
 \end{proof}
 \begin{theorem} \label{theorem 3.2}
Let $p$ and $q$ be two distinct odd primes, where $q > p$ $\&$ $q=2p-1$. Suppose $R=\mathbb{Z}_n$, where $n=pq$, then  $fdim(BS(\Gamma(R))=q$.
\end{theorem}
\begin{proof}
\noindent Consider a ring $R =\mathbb{Z}_n$, where $n=pq$ and $q=2p-1$. In particular, We have partitioned the set $V(BS(\Gamma(R))$ into $4$ sets namely, $A=\{r_1,r_2,\ldots,r_{q-1}\}$,
      \( S = \bigcup \limits_{\tau=1}^{\frac{p-1}{2}} S^\tau \),
      \( T = \bigcup \limits_{\tau=1}^{\frac{p-1}{2}} T^\tau \), and
      $U=\{u_1,u_2, \ldots, u_{p-1}\}$.  These sets are illustrated in the Figure. \ref{fig:4}.\\
      
  \noindent Let $W$ be a metric basis of the graph $BS(\Gamma(R))$. If $F = W \cup \{s\}$ is considered a fault-tolerant resolving set for any $s \in V(BS(\Gamma(R))) \setminus W$, then we encounter the following contradictions:

\begin{enumerate}
    \item If a vertex \( s \) is chosen from the set \( U \), and any vertex from the set \( W \), such as \( w_1 \), is removed (due to the property of FTRS), then according to Lemma \ref{lemma1}, the set \( F \setminus \{ w_1 \} \) is not a metric basis because it contains a vertex from \( U \). 
    \item If a vertex $s$ is chosen from set $A$, and any vertex of $W$, say $w_1 \notin A$ is removed (due to the property of FTRS), then according to Lemma \ref{lemma1}, the set \( F \setminus \{ w_1 \} \) is not a metric basis. This is because it contains two vertices of $A$. 
    \item If a vertex \( s \) is chosen from set \( T \) (or $S$), and one vertex of \( W \), say \( a_i \in A \) is removed (due to the property FTRS), then according to Lemma \ref{lemma1}, the set \( F \setminus \{ a_i \} \) is not a metric basis. This is because  
 \( F \setminus \{ a_i \} \) does not contain any vertex of \( A \). 
\end{enumerate}
From the above observation, it is follows that, if \( F \setminus \{ w \} \), where \( w \in W \), is not a metric basis. Then \( F \setminus \{ w \} \) is also not a resolving set, since \( |F \setminus \{ w \}| = |W| \). 
Consequently, we conclude that \( F \) is not an FTRS. This implies that $fdim(BS(\Gamma(\mathbb{Z_{pq}}))>q-1$ when $q=2p-1$. We now need to prove that its FTMD is $q$. We must demonstrate that $E$ is a minimum FTRS consisting of $q$ vertices, where $E=\{ a _1, t^1_2,t^1_3,t^2_4,t^2_5, \ldots,t^\frac{p-1}{2}_{p-1},t^\frac{p-1}{2}_{p}, s^1_{p+1}, s^1_{p+2}, \ldots, s^\frac{p-5}{2}_{2p-6},\\
  s^\frac{p-5}{2}_{2p-5},s^\frac{p-3}{2}_{2p-4}, s^\frac{p-3}{2}_{2p-3}, s^\frac{p-1}{2}_{2p-2}, \ldots ,s^\frac{p-1}{2}_{q-1}, u_{p-1} \}$. \\

 The metric coordinates for the vertices $\{r_ \gamma:2\leq \noindent \gamma\leq q-1\}$ are given below
 \begin{equation*}
               \delta(r_{\gamma}|E)=
                   \begin{cases}
                   (4_{\{1\}},1_{\{ 2 \}},3_{\{3\}},\ldots,3_{\{q-3\}},3_{\{q-2\}}, 3_{\{q-3\}},3_{\{q-1\}}, 2_{\{q\}}) &\textit{if}\quad \gamma= 2;\\
                  
                    (4_{\{1\}},3_{\{2\}},\ldots,3_{\{\gamma -1\}},1_{\{ \gamma \}},3_{\{\gamma +1\}},\ldots,3_{\{q-2\}},3_{\{q-1\}}, 2_{\{q\}}) &\textit{if}\quad 3\leq \gamma $ $\leq q-3  ;\\
                    (4_{\{1\}},3_{\{2\}},\ldots,3_{\{q-3\}},1_{\{q-2\}}, 3_{\{q-1\}}, 2_{\{q\}}) &\textit{if}\quad \gamma= q-2;\\
                  
                    (4_{\{1\}},3_{\{2\}},3_{\{3\}},\ldots,3_{\{q-3\}},3_{\{q-2\}}, 1_{\{q-1\}}, 2_{\{q\}} ) &\textit{if}\quad \gamma= q-1.\\
                \end{cases}
                     \end{equation*}
\noindent The metric coordinates for the vertices 
                     $\{t^1_\gamma: 1\leq \gamma\leq q-1 \}$ are given below
                        \begin{equation*}
                        \delta(t^{1}_{\gamma}|E)=
                   \begin{cases}
                   (1_{\{1\}},2_{\{2\}},2_{\{3\}},4_{\{4\}},\ldots,4_{\{q-2\}}, 4_{\{q-1\}}, 3_{\{q\}})&\textit{if}\quad\gamma= 1;\\
                     (3_{\{1\}},2_{\{2\}}, 2_{\{3 \}},2_{\{ 4 \}}, 4_{\{ 5 \}},\ldots,4_{\{q-2\}}, 4_{\{q-1\}}, 3_{\{q\}})&\textit{if}\quad  \gamma= 4;\\
                     (3_{\{1\}},2_{\{2\}}, 2_{\{3 \}},4_{\{ 4 \}}, \ldots, 4_{\{\gamma-1 \}},2_{\{\gamma \}}, 4_{\{\gamma+1 \}},\ldots,4_{\{q-2\}}, 4_{\{q-1\}}, 3_{\{q\}})&\textit{if}\quad 5 \leq  \gamma  \leq q-3  ;\\
                     (3_{\{1\}},2_{\{2\}}, 2_{\{3 \}},4_{\{ 4 \}},\ldots, 4_{\{q-3\}},2_{\{q-2\}}, 4_{\{q-1\}}, 3_{\{q\}})&\textit{if}\quad \gamma= q-2;\\
                    (3_{\{1\}},2_{\{2\}},,2_{\{3\}}, 4_{\{4\}},\ldots,4_{\{q-2\}}, 2_{\{q-1\}}, 3_{\{q\}}) &\textit{if}\quad\gamma= q-1.\\\\
                    
                    \end{cases}
                    \end{equation*}

                    \noindent The metric coordinates for the vertices 
                     $\{t^2_\gamma: 1\leq \gamma\leq q-1 \}$ are given below
                        \begin{equation*}
                        \delta(t^{2}_{\gamma}|E)=
                   \begin{cases}
                   (1_{\{1\}},4_{\{2\}},4_{\{3\}},2_{\{4\}}, 2_{\{5\}}, 4_{\{6\}}, \ldots,4_{\{q-2\}}, 4_{\{q-1\}}, 3_{\{q\}})&\textit{if}\quad\gamma= 1;\\
                   (3_{\{1\}},2_{\{2\}},4_{\{3\}},2_{\{4\}}, 2_{\{5\}}, 4_{\{6\}},\ldots,4_{\{q-2\}}, 4_{\{q-1\}}, 3_{\{q\}})&\textit{if}\quad\gamma= 2;\\
                     (3_{\{1\}},4_{\{2\}},2_{\{3\}},2_{\{4\}}, 2_{\{5\}}, 4_{\{6\}},\ldots,4_{\{q-2\}}, 4_{\{q-1\}}, 3_{\{q\}})&\textit{if}\quad\gamma= 3;\\
                     (3_{\{1\}},4_{\{2\}},4_{\{3\}},2_{\{4\}}, 2_{\{5\}}, 2_{\{6\}}, 4_{\{7\}}, \ldots,4_{\{q-2\}}, 4_{\{q-1\}},3_{\{q\}})&\textit{if}\quad\gamma= 6;\\
                     (3_{\{1\}},4_{\{2\}}, 4_{\{3 \}},2_{\{ 4 \}}, 2_{\{ 5 \}}, 4_{\{ 6 \}}, \ldots, 4_{\{\gamma-1 \}},2_{\{\gamma \}}, 4_{\{\gamma+1 \}},\ldots,4_{\{q-2\}}, 4_{\{q-1\}}, 3_{\{q\}})&\textit{if}\quad 7 \leq  \gamma  \leq q-3  ;\\
                      (3_{\{1\}},4_{\{2\}},4_{\{3\}},2_{\{4\}}, 2_{\{5\}}, 4_{\{6\}}, \ldots, 4_{\{q-3\}} ,2_{\{q-2\}}, 4_{\{q-1\}}, 3_{\{q\}})&\textit{if}\quad\gamma= q-2;\\
                    (3_{\{1\}},4_{\{2\}},4_{\{3\}},2_{\{4\}}, 2_{\{5\}}, 4_{\{6\}},\ldots,4_{\{q-3\}} ,4_{\{q-2\}}, 2_{\{q-1\}}, 3_{\{q\}})&\textit{if}\quad\gamma= q-1.\\
                    \end{cases}
                    \end{equation*}
                 \noindent The metric coordinates for the vertices 
                     $\{t^\tau_\gamma:3\leq \tau\leq \frac{p-1}{2}$  $\&$  $1\leq \gamma\leq q-1 \}$ are given below
                        \begin{equation*}
                        \delta(t^{\tau}_{\gamma}|E)=
                   \begin{cases}
                   (1_{\{1\}},4_{\{2\}},\ldots, 4_{\{2\tau-1\}},2_{\{2\tau\}},2_{\{2\tau+1\}},4_{\{2\tau+2\}},\ldots,4_{\{q-2\}}, 4_{\{q-1\}}, 3_{\{q\}})&\\
                   \qquad \qquad \qquad\qquad\qquad\qquad\qquad\qquad\qquad\qquad\qquad\qquad\qquad\qquad\qquad \qquad \quad\textit{if}\quad\gamma= 1;\\
                     (3_{\{1\}},2_{\{ 2 \}},4_{\{3\}}, \ldots, 4_{\{2\tau-1\}},2_{\{2\tau\}},2_{\{2\tau+1\}},4_{\{2\tau+2\}},\ldots,4_{\{q-2\}}, 4_{\{q-1\}}, 3_{\{q\}})&\\
                   \qquad \qquad \qquad\qquad\qquad\qquad\qquad\qquad\qquad\qquad\qquad\qquad\qquad\qquad\qquad \qquad \quad\textit{if}\quad \gamma= 2;\\
                   (3_{\{1\}},4_{\{2\}}, \ldots,4_{\{\gamma-1 \}},2_{\{ \gamma \}},4_{\{\gamma +1\}}, \ldots, 4_{\{2\tau-1\}},2_{\{2\tau\}},2_{\{2\tau+1\}},4_{\{2\tau+2\}},\ldots,4_{\{q-2\}}, 4_{\{q-1\}}, 3_{\{q\}})&\\
                   \qquad \qquad \qquad\qquad\qquad\qquad\qquad\qquad\qquad\qquad\qquad\qquad\qquad\qquad\qquad \qquad \quad\textit{if}\quad  3 \leq  \gamma  \leq 2\tau-2  ;\\
                      (3_{\{1\}},4_{\{2\}}, \ldots,4_{\{2\tau-2 \}},2_{\{ 2\tau-1 \}},2_{\{2\tau\}},2_{\{2\tau+1\}},4_{\{2\tau+2\}},\ldots,4_{\{q-2\}}, 4_{\{q-1\}}, 3_{\{q\}})&\\
                   \qquad \qquad \qquad\qquad\qquad\qquad\qquad\qquad\qquad\qquad\qquad\qquad\qquad\qquad\qquad \qquad \quad\textit{if}\quad\gamma= 2\tau-1 ;\\
                     (3_{\{1\}},4_{\{2\}}, \ldots,4_{\{2\tau-1\}},2_{\{2\tau\}},2_{\{2\tau+1\}},2_{\{ 2\tau+2 \}}, 4_{\{2\tau+3 \}},\ldots,4_{\{q-2\}}, 4_{\{q-1\}}, 3_{\{q\}})&\\
                   \qquad \qquad \qquad\qquad\qquad\qquad\qquad\qquad\qquad\qquad\qquad\qquad\qquad\qquad\qquad \qquad \quad\textit{if}\quad  \gamma= 2\tau+2 ;\\
                    (3_{\{1\}},4_{\{2\}}, \ldots,4_{\{2\tau-1\}},2_{\{2\tau\}},2_{\{2\tau+1\}},4_{\{2\tau+2\}},\ldots,4_{\{\gamma-1 \}},2_{\{ \gamma \}},4_{\{\gamma +1\}},\ldots,4_{\{q-2\}}, 4_{\{q-1\}}, 3_{\{q\}})&\\
                   \qquad \qquad \qquad\qquad\qquad\qquad\qquad\qquad\qquad\qquad\qquad\qquad\qquad\qquad\qquad \qquad \quad\textit{if}\quad  2\tau+3 \leq  \gamma  \leq q-3  ;\\
                     (3_{\{1\}},4_{\{2\}}, \ldots,4_{\{2\tau-1\}},2_{\{2\tau\}},2_{\{2\tau+1\}},4_{\{2\tau+2\}},\ldots,4_{\{q-3 \}},2_{\{ q-2 \}}, 4_{\{q-1\}}, 3_{\{q\}})&\\
                   \qquad \qquad \qquad\qquad\qquad\qquad\qquad\qquad\qquad\qquad\qquad\qquad\qquad\qquad\qquad \qquad \quad\textit{if}\quad  \gamma=  q-2  ;\\
                    (3_{\{1\}},4_{\{2\}}, \ldots,4_{\{2\tau-1\}},2_{\{2\tau\}},2_{\{2\tau+1\}},4_{\{2\tau+2\}},\ldots,4_{\{q-2\}}, 2_{\{q-1\}}, 3_{\{q\}}) &\\
                   \qquad \qquad \qquad\qquad\qquad\qquad\qquad\qquad\qquad\qquad\qquad\qquad\qquad\qquad\qquad \qquad \quad\textit{if}\quad\gamma= q-1.\\
                    \end{cases}
                    \end{equation*}
                    \noindent The metric coordinates for the vertices 
                     $\{u_\tau:1\leq \tau\leq {p-1} \}$ are given below
                        \begin{equation*}
                        \delta(u_{\tau}|E)=
                   \begin{cases}
                   (2_{\{1\}},1_{\{2\}}, 1_{\{3\}}, 3_{\{4\}},\ldots,3_{\{q-2\}}, 3_{\{q-1\}}, 4_{\{q\}})&\textit{if}\quad  \tau=1 ;\\
                   (2_{\{1\}},3_{\{2\}}, \ldots, 3_{\{2\tau-1\}},1_{\{2\tau\}},1_{\{2\tau+1\}},3_{\{2\tau+2\}},\ldots,3_{\{q-2\}}, 3_{\{q-1\}}, 4_{\{q\}})&\textit{if}\quad  2 \leq \tau  \leq p-2 ;\\
                     (2_{\{1\}},3_{\{2\}}, \ldots, 3_{\{2\tau-1\}},1_{\{2\tau\}},1_{\{2\tau+1\}},\ldots,1_{\{q-1\}}, 0_{\{q\}})&\textit{if}\quad  \tau= p-1 .\\
                     
                     \end{cases}
                     \end{equation*}
\noindent The metric coordinates for the vertices 
                     $\{s^\tau_\gamma:1\leq \tau\leq \frac{p-3}{2} $ $ \& $  $1\leq \gamma\leq q-1 \}$ are given below
                        \begin{equation*}
                        \delta(s^{\tau}_{\gamma}|E)=
                   \begin{cases}
                   (1_{\{1\}},4_{\{2\}},\ldots,4_{\{2\tau+p-2\}},2_{\{p+2\tau-1\}},2_{\{p+2\tau\}},4_{\{p+2\tau+1\}},\ldots,4_{\{q-2\}}, 4_{\{q-1\}}, 3_{\{q\}})&\\
                   \qquad \qquad \qquad\qquad\qquad\qquad\qquad\qquad\qquad\qquad\qquad\qquad\qquad\qquad\qquad \qquad \quad \textit{if}\quad\gamma= 1;\\
                    (3_{\{1\}},2_{\{2 \}},4_{\{3\}},\ldots,4_{\{2\tau+p-2\}},2_{\{p+2\tau-1\}},2_{\{p+2\tau\}},4_{\{p+2\tau+1\}},\ldots,4_{\{q-2\}}, 4_{\{q-1\}},3_{\{q\}})&\\
                   \qquad \qquad \qquad\qquad\qquad\qquad\qquad\qquad\qquad\qquad\qquad\qquad\qquad\qquad\qquad \qquad \quad \textit{if}\quad   \gamma= 2;\\
                     
                   (3_{\{1\}},4_{\{2\}}, \ldots,4_{\{\gamma-1 \}},2_{\{ \gamma \}},4_{\{\gamma +1\}},\ldots,4_{\{2\tau+p-2\}},2_{\{p+2\tau-1\}},2_{\{p+2\tau\}},4_{\{p+2\tau+1\}},\ldots,4_{\{q-2\}}, 4_{\{q-1\}}, 3_{\{q\}})&\\
                   \qquad \qquad \qquad\qquad\qquad\qquad\qquad\qquad\qquad\qquad\qquad\qquad\qquad\qquad\qquad \qquad \quad \textit{if} \quad  3 \leq  \gamma  \leq 2\tau+p-3 ;\\

                     (3_{\{1\}},4_{\{2\}}, \ldots,4_{\{\gamma-1 \}},2_{\{ 2\tau+p-2 \}},2_{\{p+2\tau-1\}},2_{\{p+2\tau\}},4_{\{p+2\tau+1\}},\ldots,4_{\{q-2\}}, 4_{\{q-1\}}, 3_{\{q\}})&\\
                   \qquad \qquad \qquad\qquad\qquad\qquad\qquad\qquad\qquad\qquad\qquad\qquad\qquad\qquad\qquad \qquad \quad\textit{if}\quad \gamma = 2\tau+p-2 ;\\
                     (3_{\{1\}},4_{\{2\}}, \ldots,4_{\{2\tau+p-2\}},2_{\{p+2\tau-1\}},2_{\{p+2\tau\}},2_{\{p+2\tau+1\}}, 4_{\{p+2\tau+2\}}, \ldots,4_{\{q-2\}}, 3_{\{q\}})&\\
                   \qquad \qquad \qquad\qquad\qquad\qquad\qquad\qquad\qquad\qquad\qquad\qquad\qquad\qquad\qquad \qquad \quad \textit{if}\quad\gamma= p+2\tau+1 ;\\
                    (3_{\{1\}},4_{\{2\}}, \ldots,4_{\{2\tau+p-2\}},2_{\{p+2\tau-1\}},2_{\{p+2\tau\}},4_{\{p+2\tau+1\}},\ldots,4_{\{\gamma-1 \}},2_{\{ \gamma \}},4_{\{\gamma +1\}},\ldots,4_{\{q-2\}}, 4_{\{q-1\}}, 3_{\{q\}})&\\
                   \qquad \qquad \qquad\qquad\qquad\qquad\qquad\qquad\qquad\qquad\qquad\qquad\qquad\qquad\qquad \qquad \quad \textit{if} \quad p+2\tau+2 \leq  \gamma  \leq q-3  ;\\
                     (3_{\{1\}},4_{\{2\}}, \ldots,4_{\{2\tau+p-2\}},2_{\{p+2\tau-1\}},2_{\{p+2\tau\}},4_{\{p+2\tau+1\}},\ldots, 4_{\{q-3\}},2_{\{q-2\}}, 4_{\{q-1\}}, 3_{\{q\}})&\\
                   \qquad \qquad \qquad\qquad\qquad\qquad\qquad\qquad\qquad\qquad\qquad\qquad\qquad\qquad\qquad \qquad \quad  \textit{if}\quad   \gamma = q-2  ;\\
                    (3_{\{1\}},4_{\{2\}}, \ldots,4_{\{2\tau+p-2\}},2_{\{p-1+2\tau\}},2_{\{p+2\tau\}},4_{\{p+2\tau+1\}},\ldots,4_{\{q-2\}}, 2_{\{q-1\}}, 3_{\{q\}}) &\\
                   \qquad \qquad \qquad\qquad\qquad\qquad\qquad\qquad\qquad\qquad\qquad\qquad\qquad\qquad\qquad \qquad \quad \textit{if}\quad\gamma= q-1.\\
                      \end{cases}
                    \end{equation*}
                    \noindent The metric coordinates for the vertices 
                     $\{s^\tau_\gamma: \tau =\frac{p-1}{2} $ $\&$  $1\leq \gamma\leq q-1 \}$ are given below
                        \begin{equation*}
                        \delta(s^{\tau}_{\gamma}|E)=
                   \begin{cases}
                   (1_{\{1\}},4_{\{2\}},\ldots,4_{\{2p-3 \}},2_{\{2p-2 \}}, \ldots,2_{\{q-3\}},2_{\{q-2\}}, 2_{\{q-1\}}, 1_{\{q\}})&\\
                   \qquad \qquad \qquad\qquad\qquad\qquad\qquad\qquad\qquad\qquad\qquad\qquad\qquad\qquad\qquad \qquad \quad\textit{if}\quad\gamma= 1;\\
                    (3_{\{1\}},2_{\{2\}},4_{\{3\}}, \ldots, 4_{\{2p-3 \}},2_{\{2p-2 \}}, \ldots,2_{\{q-3\}},2_{\{q-2\}}, 2_{\{q-1\}}, 1_{\{q\}}) &\\
                   \qquad \qquad \qquad\qquad\qquad\qquad\qquad\qquad\qquad\qquad\qquad\qquad\qquad\qquad\qquad \qquad \quad\textit{if}\quad\gamma= 2;\\ 
                   (3_{\{1\}},4_{\{2\}}, \ldots,4_{\{\gamma-1\}},2_{\{ \gamma\}},4_{\{\gamma +1\}}, \ldots, 4_{\{2p-3 \}},2_{\{2p-2 \}}, \ldots,2_{\{q-3\}},2_{\{q-2\}}, 2_{\{q-1\}}, 1_{\{q\}}) &\\
                   \qquad \qquad \qquad\qquad\qquad\qquad\qquad\qquad\qquad\qquad\qquad\qquad\qquad\qquad\qquad \qquad \quad\textit{if}\quad   3 \leq  \gamma  \leq 2p-4  ;\\
                      (3_{\{1\}},4_{\{2\}}, \ldots, 4_{\{2p-4 \}} 2_{\{2p-3 \}},2_{\{2p-2 \}}, \ldots,2_{\{q-3\}},2_{\{q-2\}}, 2_{\{q-1\}}, 1_{\{q\}}) &\\
                   \qquad \qquad \qquad\qquad\qquad\qquad\qquad\qquad\qquad\qquad\qquad\qquad\qquad\qquad\qquad \qquad \quad\textit{if}\quad   \gamma= 2p-3;\\
                    (3_{\{1\}},4_{\{2\}}, \ldots, 4_{\{2p-3\}},2_{\{2p-2\}}, \ldots,2_{\{q-3\}},2_{\{q-2\}}, 2_{\{q-1\}}, 1_{\{q\}}) &\\
                   \qquad \qquad \qquad\qquad\qquad\qquad\qquad\qquad\qquad\qquad\qquad\qquad\qquad\qquad\qquad \qquad \quad\textit{if}\quad\gamma= q-1.\\
                     \end{cases}
                    \end{equation*}

                    \noindent After analyzing the above codes, it is clear that removing any vertices in set $E$ (i.e., $E-\{a\}$, for any $a\in E$ ) while ensuring that all vertices have unique metric coordinates with respect to $E-\{a\}$ preserves the property of the fault- tolerance in $E$. This implies that $E$ is an FTRS, leading to $fdim(BS\Gamma(G)) \leq q$. Thus $fdim(BS\Gamma(G))=q$.   \\
                    \end{proof}
                    \begin{exm}
\noindent  {{ For $p=7$ $\&$ $q=13$ }}\\
\noindent Consider a ring $R =\mathbb{Z}_{91}$. We define $BS(\Gamma(\mathbb{Z}_{91}))$ as the barycentric subdivision of $\Gamma(\mathbb{Z}_{91})$. The graph $BS(\Gamma(\mathbb{Z}_{91}))$ consists of  $90$ vertices and $144$ edges. In particular, we have eight sets of vertices: $A=\{r_1,r_2,\ldots,r_{12}\}$, $S^1=\{s^1_1,s^1_2,\ldots,s^1_{12}\}$, $S^2=\{s^2_1,s^2_2,s^2_3,\ldots,s^2_{12} \}$, 
$S^3=\{s^3_1,s^3_2,s^3_3,\ldots,s^3_{12} \}$, $T^1=\{t^1_1,t^1_2,\ldots,t^1_{12}\}$, $T^2=\{t^2_1,t^2_2,t^2_3,\ldots,t^2_{12} \}$,  
   $T^3=\{t^3_1,t^3_2,t^3_3,\ldots,t^3_{12} \}$  and $U=\{u_1,u_2, u_3, u_4, u_5,u_6\}$. These sets are illustrated in Figure.\ref{fig:10}.\\
    
\begin{figure}[ht!]
    \centering
    \includegraphics[width=10cm]{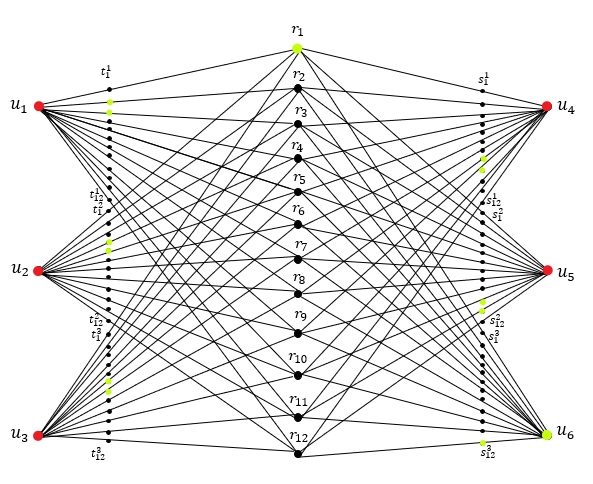}
    \caption{Barycentric Subdivision of Zero Divisor Graph $\mathbb{Z}_{91}$}
    \label{fig:10}
\end{figure}

\noindent Let us consider a set $E=\{ r_1, t^1_2,t^1_3,t^2_4,t^2_5, t^3_6,t^3_7, s^1_{8},s^1_{9}, s^2_{10}, s^2_{11}, s^3_{12}, u_{6}\}$. 
Then, to show that $E$  is an FTRS of $BS(\Gamma(R))$, we need to assign unique metric coordinates for every vertex of $V(BS(\Gamma(R)))-E$ with respect to set $E-\{a\}$, $\forall~a\in E$.\\ 

\noindent Metric coordinates for every vertex of $V(BS(\Gamma(R)))-E$ with respect to set $E$ are given below 
\begin{table}[H]
    \centering
    \begin{tabular}{|p{0.35\linewidth}|p{0.35\linewidth}
    |}
      \hline
        $\delta(t^3_1|E)=(1,4,4,4,4,2,2,4,4,4,4,4,3)$ &  $\delta(s^1 _1|E)=(1,4,4,4,4,4,4,2,2,4,4,4,3)$\\ 
       $\delta(t^3 _2|E)= (3,2,4,4,4,2,2,4,4,4,4,4,3)$ &  $\delta(s^1 _2|E)=(3,2,4,4,4,4,4,2,2,4,4,4,3)$\\ 
       $\delta(t^3 _3|E)= (3,4,2,4,4,2,2,4,4,4,4,4,3)$ &  $\delta(s^1_3|E)=(3,4,2,4,4,4,4,2,2,4,4,4,3)$  \\
        $\delta(t^3 _4|E)= (3,4,4,2,4,2,2,4,4,4,4,4,3)$ &  $\delta(s^1 _4|E)=(3,4,4,2,4,4,4,2,2,4,4,4,3)$ \\
       $\delta(t^3 _5|E)= (3,4,4,4,2,2,2,4,4,4,4,4,3)$ & $\delta(s^2 _5|E)=(3,4,4,4,2,4,4,2,2,4,4,4,3)$  \\
    $\delta(t^3_6|E)= (3,4,4,4,4,0,2,4,4,4,4,4,3)$ &  $\delta(s^1 _6|E)=(3,4,4,4,4,2,4,2,2,4,4,4,3)$   \\
      $\delta(t^3_7|E)= (3,4,4,4,4,2,0,4,4,4,4,4,3)$ &  $\delta(s^1 _7|E)=(3,4,4,4,4,4,2,2,2,4,4,4,3)$ \\
     $\delta(t^3_8|E)= (3,4,4,4,4,2,2,2,4,4,4,4,3)$ &  $\delta(s^1 _8|E)=(3,4,4,4,4,4,4,0,2,4,4,4,3)$ \\
     $\delta(t^3_9|E)= (3,4,4,4,4,2,2,4,2,4,4,4,3)$ &  $\delta(s^1 _9|E)=(3,4,4,4,4,4,4,2,0,4,4,4,3)$   \\
   $\delta(t^3_{10}|E)= $$(3,4,4,4,4,2,2,4,4,2,4,4,3)$ &  $\delta(s^1 _{10}|E)=(3,4,4,4,4,4,4,2,2,2,4,4,3)$   \\
   $\delta(t^3_{11}|E)= (3,4,4,4,4,2,2,4,4,4,2,4,3)$&$\delta(s^1 _{10}|E)=(3,4,4,4,4,4,4,2,2,4,2,4,3)$ \\
   $\delta(t^3_{12}|E)= (3,4,4,4,4,2,2,4,4,2,4,2,3)$ & $\delta(s^1 _{10}|E)=(3,4,4,4,4,4,4,2,2,4,4,2,3)$ \\
      \hline
      \end{tabular}
      
    \end{table}
    \begin{table}[H]
    \centering
    \begin{tabular}{|p{0.35\linewidth}|p{0.35\linewidth}
    |}
       \hline
        $\delta(s^2_1|E)=(1,4,4,4,4,4,4,4,4,2,2,4,3)$ &  $\delta(s^3 _1|E)=(1,4,4,4,4,4,4,4,4,4,4,2,1)$\\ 
       $\delta(s^2 _2|E)= (3,2,4,4,4,4,4,4,4,2,2,4,3)$ &  $\delta(s^3 _2|E)=(3,2,4,4,4,4,4,4,4,4,4,2,1)$\\ 
       $\delta(s^2 _3|E)= (3,4,2,4,4,4,4,4,4,2,2,4,3)$ &  $\delta(s^3 _3|E)=(3,4,2,4,4,4,4,4,4,4,4,2,1)$  \\
        $\delta(s^2 _4|E)= (3,4,4,2,4,4,4,4,4,2,2,4,3)$ &  $\delta(s^3 _4|E)=(3,4,4,2,4,4,4,4,4,4,4,2,1)$ \\
       $\delta(s^2_5|E)= (3,4,4,4,2,4,4,4,2,2,4,4,3)$ & $\delta(s^3 _5|E)=(3,4,4,4,2,4,4,4,4,4,4,2,1)$  \\
    $\delta(s^2_6|E)= (3,4,4,4,4,2,4,4,4,4,2,2,3)$ &  $\delta(s^3 _6|E)=(3,4,4,4,4,2,4,4,4,4,4,2,1)$   \\
      $\delta(s^2_7|E)= (3,4,4,4,4,4,2,4,4,2,2,4,3)$ &  $\delta(s^3 _7|E)=(3,4,4,4,4,4,2,4,4,4,4,2,1)$ \\
     $\delta(s^2_8|E)= (3,4,4,4,4,4,4,2,4,2,2,4,3)$ &  $\delta(s^3 _8|E)=(3,4,4,4,4,4,4,2,4,4,4,2,1)$ \\
     $\delta(s^2_9|E)= (3,4,4,4,4,4,4,4,2,2,2,4,3)$ &  $\delta(s^3 _9|E)=(3,4,4,4,4,4,4,4,2,4,4,2,1)$   \\
   $\delta(s^2_{10}|E)= $$(3,4,4,4,4,4,4,4,4,0,2,4,3)$ &  $\delta(s^3 _{10}|E)=(3,4,4,4,4,4,4,4,4,2,4,2,1)$   \\
   $\delta(s^2_{11}|E)= (3,4,4,4,4,4,4,4,4,2,0,4,3)$&$\delta(s^3 _{10}|E)=(3,4,4,4,4,4,4,4,4,4,2,2,1)$ \\
   $\delta(s^3_{12}|E)= (3,4,4,4,4,4,4,4,4,2,2,2,3)$ & $\delta(s^2 _{10}|E)=(3,4,4,4,4,4,4,4,4,4,4,0,1)$ \\
      \hline
      \end{tabular}
    \end{table}
    \begin{table}[H]
    \centering
    \begin{tabular}{|p{0.35\linewidth}|p{0.35\linewidth}|}
      \hline
       $\delta(t^1 _1|E)=(1,2,2,4,4,4,4,4,4,4,4,4,3)$& $\delta(t^2 _1|E)=(1,4,4,2,2,4,4,4,4,4,4,4,3)$\\
      $\delta(t^1 _2|E)= (3,0,2,4,4,4,4,4,4,4,4,4,3)$& $\delta(t^2 _2|E)=(3,2,4,2,2,4,4,4,4,4,4,4,3
      )$\\
      $\delta(t^1 _3|E)= (3,2,0,4,4,4,4,4,4,4,4,4,3)$& $\delta(t^2 _3|E)=(3,4,2,2,2,4,4,4,4,4,4,4,3)$\\
      $\delta(t^1 _4|E)= (3,2,2,2,4,4,4,4,4,4,4,4,3)$& $\delta(t^2 _4|E)=(3,4,4,0,2,4,4,4,4,4,4,4,3)$\\
      $\delta(t^1 _5|E)= (3,2,2,4,2,4,4,4,4,4,4,4,3)$& $\delta(t^2 _5|E)=(3,4,4,2,0,4,4,4,4,4,4,4,3)$\\
      $\delta(t^1 _6|E)= (3,2,2,4,4,2,4,4,4,4,4,4,3)$ & $\delta(t^2 _6|E)=(3,4,4,2,2,2,4,4,4,4,4,4,3)$\\
      $\delta(t^1 _7|E)= (3,2,2,4,4,4,2,4,4,4,4,4,3)$& $\delta(t^2 _7|E)=(3,4,4,2,2,4,2,4,4,4,4,4,3)$\\
      $\delta(t^1 _8|E)= (3,2,2,4,4,4,4,2,4,4,4,4,3)$& $\delta(t^2 _8|E)=(3,4,4,2,2,4,4,2,4,4,4,4,3)$\\
      $\delta(t^1 _9|E)= (3,2,2,4,4,4,4,4,2,4,4,4,3)$& $\delta(t^2 _9|E)=(3,4,4,2,2,4,4,4,2,4,4,4,3)$\\
      $\delta(t^1 _{10}|E)= $$(3,2,2,4,4,4,4,4,4,2,4,4,3)$& $\delta(t^2 _{10}|E)=(3,4,4,2,2,4,4,4,4,2,4,4,3)$\\
      $\delta(t^1 _{10}|E)= $$(3,2,2,4,4,4,4,4,4,4,2,4,3)$&$\delta(t^2 _{11}|E)=(3,4,4,2,2,4,4,4,4,4,2,4,3)$\\
      $\delta(t^1 _{12}|E)= $$(3,2,2,4,4,4,4,4,4,4,4,2,3)$&$\delta(t^2 _{12}|E)=(3,4,4,2,2,4,4,4,4,4,4,2,3)$\\
      \hline
    \end{tabular}   
    \end{table} 
    \begin{table}[h!]
    \centering
    \begin{tabular}{|p{0.35\linewidth}|p{0.35\linewidth}|}      
      \hline
   $\delta(r^1 _1|E)=(0,3,3,3,3,3,3,3,3,3,3,3,2)$ &  $ \delta(u_1|E)=(2,1,1,3,3,3,3,3,3,3,3,3,3)$\\
  $\delta(r^1 _2|E)= (4,1,3,3,3,3,3,3,3,3,3,3,2) $ &  $ \delta(u_2|E)=(2,3,3,1,1,3,3,3,3,3,3,3,3)$\\
  $\delta(r^1 _3|E)= (4,3,1,3,3,3,3,3,3,3,3,3,2)$ &$ \delta(u_3|E)=(2,3,3,3,3,1,1,3,3,3,3,3,3)$\\
$\delta(r^1 _4|E)= (4,3,3,1,3,3,3,3,3,3,3,3,2)$ &$ \delta(u_4|E)=(2,3,3,3,3,3,3,1,1,3,3,3,3)$\\
  $\delta(r^1 _5|E)= (4,3,3,3,1,3,3,3,3,3,3,3,2)$ & $ \delta(u_5|E)=(2,3,3,3,3,3,3,3,3,1,1,3,3)$ \\
    $\delta(r^1 _6|E)= (4,3,3,3,3,1,3,3,3,3,3,3,2)$ & $ \delta(u_6|E)=(2,3,3,3,3,3,3,3,3,3,3,1,0)$ \\
 $\delta(r^1 _7|E)= (4,3,3,3,3,3,1,3,3,3,3,3,2)$ &  \\
 $\delta(r^1 _8|E)= (4,3,3,3,3,3,3,1,3,3,3,3,2)$   & \\
   $\delta(r^1 _9|E)= (4,3,3,3,3,3,3,3,1,3,3,3,2)$  &  \\
 $\delta(r^1 _{10}|E)= (4,3,3,3,3,3,3,3,3,1,3,3,2)$   & \\
 $\delta(r^1 _{11}|E)= (4,3,3,3,3,3,3,3,3,3,1,3,2)$   & \\
 $\delta(r^1 _{12}|E)= (4,3,3,3,3,3,3,3,3,3,3,1,2)$   & \\
      \hline
      \end{tabular}
\label{tab:example3}
\end{table}
\noindent Based on the above codes, it is evident that if we remove any single vertex from $E$ (i.e., $E-\{a\}$, for any  $a\in E $), then all vertices with unique metric coordinates with respect to $E-\{a\}$ indicate that $E$ is an FTRS. According to Theorem \ref{theorem 3.2}, it is evident that $fdim(BS(\Gamma(\mathbb{Z}_{91}))) = 13$ ($\because$ it is in the form $2p-1$). Therefore, the set $E$ is the minimum FTRS for $BS(\Gamma(\mathbb{Z}_{91}))$.
\end{exm}
\begin{theorem}\label{z}
Let $p$ and $q$ be two distinct odd primes, where $q>p$ $\&$ $q = 2p-3$. Suppose $R=\mathbb{Z}_n$, where $n=pq$, then,
$fdim(BS(\Gamma(R))
)=q$.
\end{theorem}
\begin{proof}
Consider a ring $R = \mathbb{Z}_n$, where $n=pq$, where $q = 2p-3$. Let $\Gamma(R)$ denote the zero divisor graph of $R$. We define $BS(\Gamma(R))$ as the barycentric subdivision of $\Gamma(R)$.\\

\noindent From the Proposition \ref{ET}, and Equation \ref{a} it is clear that $fdim(BS(\Gamma(R))
)\geq q$, We now need to prove that its FTMD is $q$. We must demonstrate that $E$ is a minimum FTRS consisting of $q$ vertices, $E=\{   a _1, t^1_2,t^1_3,t^2_4,t^2_5,\ldots,t^\frac{p-1}{2}_{p-1},t^\frac{p-1}{2}_{p}, s^1_{p+1}, \\
s^1_{p+2},\ldots, s^\frac{p-5}{2}_{2p-6},s^\frac{p-5}{2}_{2p-5},s^\frac{p-3}{2}_{2p-4},
\ldots,s^\frac{p-3}{2}_{q-1}, s^\frac{p-1}{2}_{q-1}\}$.\\ 

 The metric coordinates for the vertices $\{r_ \gamma:1\leq \noindent \gamma\leq q-1\}$ are given below
 \begin{equation*}
               \delta(r_{\gamma}|E)=
                   \begin{cases}
                   (4_{\{1\}},1_{\{ 2 \}},3_{\{3\}},\ldots,3_{\{q-2\}},3_{\{q-1\}}, 3_{\{q\}}) &\textit{if}\quad \gamma =2  ;\\
                    (4_{\{1\}},3_{\{2\}},\ldots,3_{\{\gamma -1\}},1_{\{ \gamma \}},3_{\{\gamma +1\}},\ldots,3_{\{q-1\}}, 3_{\{q\}}) &\textit{if}\quad 3\leq \gamma $ $\leq q-2  ;\\
                    (4_{\{1\}},3_{\{2\}},\ldots, 3_{\{q-2\}}),1_{\{q-1\}}, 1_{\{q\}}) &\textit{if}\quad  \gamma = q-1  .\\
                    \end{cases}
                     \end{equation*}
                     
                    \noindent The metric coordinates for the vertices 
                     $\{s^\tau_\gamma:1\leq \tau\leq \frac{p-5}{2} $ $ \& $  $1\leq \gamma\leq q-1 \}$ are given below
                        \begin{equation*}
                        \delta(s^{\tau}_{\gamma}|E)=
                   \begin{cases}
                   (1_{\{1\}},4_{\{2\}},\ldots,4_{\{2\tau+p-2\}},2_{\{p+2\tau-1\}},2_{\{p+2\tau\}},4_{\{p+2\tau+1\}},\ldots,4_{\{q-1\}}, 4_{\{q\}})&\\
                   \qquad \qquad \qquad\qquad\qquad\qquad\qquad\qquad\qquad\qquad\qquad\qquad\qquad\qquad\qquad \qquad \quad \textit{if}\quad \gamma= 1;\\
         (3_{\{1\}}, 2_{\{ 2\}},4_{\{3\}},\ldots,4_{\{2\tau+p-2\}},2_{\{p+2\tau-1\}},2_{\{p+2\tau\}},4_{\{p+2\tau+1\}},\ldots,4_{\{q-1\}}, 4_{\{q\}})&\\
                   \qquad \qquad \qquad\qquad\qquad\qquad\qquad\qquad\qquad\qquad\qquad\qquad\qquad\qquad\qquad \qquad \quad \textit{if}\quad \gamma=2 ;\\
                     (3_{\{1\}},4_{\{2\}}, \ldots,4_{\{\gamma-1 \}},2_{\{ \gamma \}},4_{\{\gamma +1\}},\ldots,4_{\{2\tau+p-2\}},2_{\{p+2\tau-1\}},2_{\{p+2\tau\}},4_{\{p+2\tau+1\}},\ldots,4_{\{q-1\}}, 4_{\{q\}})
                     &\\
                   \qquad \qquad \qquad\qquad\qquad\qquad\qquad\qquad\qquad\qquad\qquad\qquad\qquad\qquad\qquad \qquad \quad \textit{if}\quad  3 \leq  \gamma  \leq 2\tau+p-3 ;\\
(3_{\{1\}},4_{\{2\}}, \ldots,4_{\{2\tau \}},2_{\{ 2\tau+p-2 \}},2_{\{p+2\tau-1\}},2_{\{p+2\tau\}},4_{\{p+2\tau+1\}},\ldots,4_{\{q-1\}}, 4_{\{q\}})&\\
                   \qquad \qquad \qquad\qquad\qquad\qquad\qquad\qquad\qquad\qquad\qquad\qquad\qquad\qquad\qquad \qquad \quad \textit{if}\quad   \gamma=2\tau+p-2 ;\\
(3_{\{1\}},4_{\{2\}}, \ldots,4_{\{2\tau+p-2\}},2_{\{p+2\tau-1\}},2_{\{p+2\tau\}},2_{\{ p+2\tau+1 \}},4_{\{ p+2\tau+2 \}},\ldots,4_{\{q-1\}}, 4_{\{q\}})&\\
                   \qquad \qquad \qquad\qquad\qquad\qquad\qquad\qquad\qquad\qquad\qquad\qquad\qquad\qquad\qquad \qquad \quad {if}\quad  \gamma=p+2\tau+1;\\
                    (3_{\{1\}},4_{\{2\}}, \ldots,4_{\{2\tau+p-2\}},2_{\{p+2\tau-1\}},2_{\{p+2\tau\}},4_{\{p+2\tau+1\}},\ldots,4_{\{\gamma-1 \}},2_{\{ \gamma \}},4_{\{\gamma +1\}},\ldots,4_{\{q-1\}}, 4_{\{q\}})&\\
                   \qquad \qquad \qquad\qquad\qquad\qquad\qquad\qquad\qquad\qquad\qquad\qquad\qquad\qquad\qquad \qquad \quad \textit{if}\quad p+2\tau+2 \leq  \gamma  \leq q-2;\\
                     (3_{\{1\}},4_{\{2\}}, \ldots,4_{\{2\tau+p-2\}},2_{\{p+2\tau-1\}},2_{\{p+2\tau\}},4_{\{p+2\tau+1\}},\ldots, 4_{\{q-2\}}, 2_{\{q-1\}}, 2_{\{q\}})&\\
                   \qquad \qquad \qquad\qquad\qquad\qquad\qquad\qquad\qquad\qquad\qquad\qquad\qquad\qquad\qquad \qquad \quad \textit{if}\quad  \gamma = q-1  .\\
                     \end{cases}
                    \end{equation*}
                    \noindent The metric coordinates for the vertices 
                     $\{s^\tau_\gamma: \tau =\frac{p-3}{2} $ $\&$  $1\leq \gamma\leq q-1 \}$ are given below
                        \begin{equation*}
                        \delta(s^{\tau}_{\gamma}|E)=
                   \begin{cases}
                   (1_{\{1\}},4_{\{2\}}, \ldots, 4_{\{2p-5 \}},2_{\{2p-4 \}}, \ldots,2_{\{q-3\}},2_{\{q-1\}}, 4_{\{q\}})&\textit{if}\quad \gamma= 1;\\
                   (3_{\{1\}},2_{\{ 2\}},4_{\{3\}}, \ldots, 4_{\{2p-5 \}},2_{\{2p-4 \}}, \ldots,2_{\{q-3\}},2_{\{q-1\}}, 4_{\{q\}})&\textit{if}\quad    \gamma =2  ;\\
                   (3_{\{1\}},4_{\{2\}}, \ldots,4_{\{\gamma-1\}},2_{\{ \gamma\}},4_{\{\gamma +1\}}, \ldots, 4_{\{2p-5 \}},2_{\{2p-4 \}}, \ldots,2_{\{q-3\}},2_{\{q-1\}}, 4_{\{q\}})&\textit{if}\quad  3 \leq  \gamma  \leq 2p-6  ;\\
                   3_{\{1\}},4_{\{2\}}, \ldots,4_{\{2p-6\}},2_{\{ 2p-5\}},2_{\{2p-4 \}}, \ldots,2_{\{q-3\}},2_{\{q-1\}}, 4_{\{q\}})&\textit{if}\quad  \gamma = 2p-5.\\
                    \end{cases}
                    \end{equation*}
\noindent The metric coordinates for the vertices 
                     $\{s^\tau_\gamma: \tau =\frac{p-1}{2} $ $\&$  $1\leq \gamma\leq q-1 \}$ are given below
                        \begin{equation*}
                        \delta(s^{\tau}_{\gamma}|E)=
                   \begin{cases}
                   (1_{\{1\}},4_{\{2\}},\ldots, 4_{\{2p-3 \}},4_{\{2p-2 \}}, \ldots,4_{\{q-3\}},4_{\{q-1\}}, 2_{\{q\}})&\textit{if}\quad\gamma= 1;\\
                   (3_{\{1\}},2_{\{ 2\}},4_{\{3\}}, \ldots,4_{\{q-1\}}, 2_{\{q\}})&\textit{if}\quad    \gamma =2;\\
                   (3_{\{1\}},4_{\{2\}}, \ldots,4_{\{\gamma-1\}},2_{\{ \gamma\}},4_{\{\gamma +1\}}, \ldots,4_{\{q-1\}}, 2_{\{q\}})&\textit{if}\quad  3 \leq  \gamma  \leq q-2  ;\\
                      (3_{\{1\}},4_{\{2\}},\ldots,4_{\{q-2\}},2_{\{q-1\}}, 0_{\{q\}})&\textit{if}\quad   \gamma= q-1  .\\
                    \end{cases}
                    \end{equation*}

\noindent The metric coordinates for the vertices 
                     $\{t^1_\gamma: 1\leq \gamma\leq q-1 \}$ are given below
                        \begin{equation*}
                        \delta(t^{1}_{\gamma}|E)=
                   \begin{cases}
                   (1_{\{1\}},2_{\{2\}},2_{\{3\}},4_{\{4\}},\ldots,4_{\{q-1\}}, 4_{\{q\}})&\textit{if}\quad\gamma= 1;\\
                     (3_{\{1\}},2_{\{2\}}, 2_{\{3 \}},2_{\{ 4 \}}, 4_{\{ 5 \}},\ldots,4_{\{q-1\}}, 4_{\{q\}})&\textit{if}\quad \gamma= 4;\\
                     (3_{\{1\}},2_{\{2\}}, 2_{\{3 \}},4_{\{ 4 \}}, \ldots, 4_{\{\gamma-1 \}},2_{\{\gamma \}}, 4_{\{\gamma+1 \}},\ldots,4_{\{q-1\}}, 4_{\{q\}})&\textit{if}\quad  5 \leq  \gamma  \leq q-2  ;\\
                     (3_{\{1\}},2_{\{2\}},2_{\{3\}} 4_{\{4\}},\ldots, 4_{\{q-2\}}2_{\{q-1\}}, 2_{\{q-1\}}) &\textit{if}\quad\gamma= q-1.\\
                \end{cases}
                    \end{equation*}

                    \noindent The metric coordinates for the vertices 
                     $\{t^2_\gamma: 1\leq \gamma\leq q-1 \}$ are given below
                        \begin{equation*}
                        \delta(t^{2}_{\gamma}|E)=
                   \begin{cases}
                   (1_{\{1\}},4_{\{2\}},4_{\{3\}},2_{\{4\}}, 2_{\{5\}}, 4_{\{6\}}, \ldots,4_{\{q-1\}}, 4_{\{q\}})&\textit{if}\quad\gamma= 1;\\
                   (3_{\{1\}},2_{\{2\}},4_{\{3\}},2_{\{4\}}, 2_{\{5\}}, 4_{\{6\}}, \ldots,4_{\{q-1\}}, 4_{\{q\}})&\textit{if}\quad\gamma= 2;\\
                     (3_{\{1\}},4_{\{2\}},2_{\{3\}},2_{\{4\}}, 2_{\{5\}}, 4_{\{6\}}, \ldots,4_{\{q-1\}}, 4_{\{q\}})&\textit{if}\quad\gamma= 3;\\
                     (3_{\{1\}},4_{\{2\}},4_{\{3\}},2_{\{4\}}, 2_{\{5\}}, 2_{\{6\}}, 4_{\{7\}}  \ldots,4_{\{q-1\}}, 4_{\{q\}})&\textit{if}\quad\gamma= 6;\\
                     
                   (3_{\{1\}},4_{\{2\}}, 4_{\{3 \}},2_{\{ 4 \}}, 2_{\{ 5 \}}, 4_{\{ 6 \}}, \ldots, 4_{\{\gamma-1 \}}2_{\{\gamma \}}, 4_{\{\gamma+1 \}},\ldots,4_{\{q-1\}}, 4_{\{q\}})&\textit{if}\quad  7 \leq  \gamma  \leq q-2  ;\\
                     (3_{\{1\}},4_{\{2\}},4_{\{3\}},2_{\{4\}}, 2_{\{5\}}, 4_{\{6\}}, \ldots 4_{\{q-2\}} ,2_{\{q-1\}}, 2_{\{q\}})&\textit{if}\quad\gamma= q-1.\\
                    \end{cases}
                    \end{equation*}

                     \noindent The metric coordinates for the vertices 
                     $\{t^\tau_\gamma:3\leq \tau\leq \frac{p-1}{2}$  $\&$  $1\leq \gamma\leq q-1 \}$ are given below
                        \begin{equation*}
                        \delta(t^{\tau}_{\gamma}|E)=
                   \begin{cases}
                   (1_{\{1\}},4_{\{2\}},\ldots, 4_{\{2\tau-1\}},2_{\{2\tau\}},2_{\{2\tau+1\}},4_{\{2\tau+2\}},\ldots,4_{\{q-1\}}, 4_{\{q\}})&\\
                   \qquad \qquad \qquad\qquad\qquad\qquad\qquad\qquad\qquad\qquad\qquad\qquad\qquad\qquad\qquad \qquad \quad  \textit{if}\quad \gamma= 1;\\
                     (3_{\{1\}},2_{\{ 2 \}},4_{\{3\}}, \ldots, 4_{\{2\tau-1\}},2_{\{2\tau\}},2_{\{2\tau+1\}},4_{\{2\tau+2\}},\ldots,4_{\{q-1\}}, 4_{\{q\}})&\\
                   \qquad \qquad \qquad\qquad\qquad\qquad\qquad\qquad\qquad\qquad\qquad\qquad\qquad\qquad\qquad \qquad \quad  \textit{if}\quad\gamma= 2;\\
                   (3_{\{1\}},4_{\{2\}}, \ldots,4_{\{\gamma-1 \}},2_{\{ \gamma \}},4_{\{\gamma +1\}}, \ldots, 4_{\{2\tau-1\}},2_{\{2\tau\}},2_{\{2\tau+1\}},4_{\{2\tau+2\}},\ldots,4_{\{q-1\}}, 4_{\{q\}})&\\
                   \qquad \qquad \qquad\qquad\qquad\qquad\qquad\qquad\qquad\qquad\qquad\qquad\qquad\qquad\qquad \qquad \quad \textit{if}\quad  3 \leq  \gamma  \leq 2\tau-2  ;\\
                      (3_{\{1\}},4_{\{2\}}, \ldots,4_{\{2\tau-2 \}},2_{\{ 2\tau-1 \}},2_{\{2\tau\}},2_{\{2\tau+1\}},4_{\{2\tau+2\}},\ldots,4_{\{q-1\}}, 4_{\{q\}})&\\
                   \qquad \qquad \qquad\qquad\qquad\qquad\qquad\qquad\qquad\qquad\qquad\qquad\qquad\qquad\qquad \qquad \quad \textit{if}\quad \gamma= 2\tau-1 ;\\
                     (3_{\{1\}},4_{\{2\}}, \ldots,4_{\{2\tau-1\}},2_{\{2\tau\}},2_{\{2\tau+1\}},2_{\{ 2\tau+2 \}}, 4_{\{2\tau+3 \}},\ldots,4_{\{q-1\}}, 4_{\{q\}})&\\
                   \qquad \qquad , \qquad\qquad\qquad\qquad\qquad\qquad\qquad\qquad\qquad\qquad\qquad\qquad\qquad \qquad \quad  \textit{if}\quad \gamma= 2\tau+2 ;\\
                    (3_{\{1\}},4_{\{2\}}, \ldots,4_{\{2\tau-1\}},2_{\{2\tau\}},2_{\{2\tau+1\}},4_{\{2\tau+2\}},\ldots,4_{\{\gamma-1 \}},2_{\{ \gamma \}},4_{\{\gamma +1\}},\ldots,4_{\{q-1\}}, 4_{\{q\}})&\\
                   \qquad \qquad \qquad\qquad\qquad\qquad\qquad\qquad\qquad\qquad\qquad\qquad\qquad\qquad\qquad \qquad \quad  \textit{if}\quad 2\tau+3 \leq  \gamma  \leq q-2  ;\\
                     % (3_{\{1\}},4_{\{2\}}, \ldots,4_{\{2\tau-1\}},2_{\{2\tau\}},2_{\{2\tau+1\}},4_{\{2\tau+2\}},\ldots,4_{\{q-3 \}},2_{\{ q-1 \}})&\\
                     % \qquad  \qquad \qquad \qquad \qquad \qquad \qquad \qquad \qquad \qquad  \qquad \qquad \qquad \qquad \qquad  \gamma=  q-2  ;\\
                    (3_{\{1\}},4_{\{2\}}, \ldots,4_{\{2\tau-1\}},2_{\{2\tau\}},2_{\{2\tau+1\}},4_{\{2\tau+2\}},\ldots, 4_{\{q-2\}}),2_{\{q-1\}}, 2_{\{q\}}) &\\
                   \qquad \qquad \qquad\qquad\qquad\qquad\qquad\qquad\qquad\qquad\qquad\qquad\qquad\qquad\qquad \qquad \quad \textit{if}\quad\gamma= q-1.\\
                    \end{cases}
                    \end{equation*}
                    
                \noindent The metric coordinates for the vertices 
                     $\{u_\tau:1\leq \tau\leq {p-1} \}$ are given below
                        \begin{equation*}
                        \delta(u_{\tau}|E)=
                   \begin{cases}
(2_{\{1\}},1_{\{2\}}, 1_{\{3\}}, 3_{\{4\}},\ldots,3_{\{q-1\}}, 3_{\{q\}})&\textit{if}\quad  \tau=1 ;\\
                   (2_{\{1\}},3_{\{2\}}, \ldots, 3_{\{2\tau-1\}},1_{\{2\tau\}},1_{\{2\tau+1\}},3_{\{2\tau+2\}},\ldots,3_{\{q-1\}}, 3_{\{q\}})&\textit{if}\quad  2 \leq \tau  \leq p-3 ;\\
                     (2_{\{1\}},3_{\{2\}}, \ldots, 3_{\{2\tau-1\}},1_{\{2\tau\}},\ldots,1_{\{q-1\}}, 3_{\{q\}})&\textit{if}\quad  \tau= p-2 ;\\
                     (2_{\{1\}},3_{\{2\}}, \ldots, 3_{\{q-1\}}, 1_{\{q\}})&\textit{if}\quad  \tau= p-1.\\
                     \end{cases}
                     \end{equation*}
                     
\noindent After analyzing the above codes, it is clear that removing any vertex in set $E$ (i.e., $E-\{a\}$, where $a\in E$ ) while ensuring that all remaining vertices have unique metric coordinates with respect to $E-\{a\}$ preserves the property of the fault- tolerance in $E$. This implies that $E$ is a FTRS, leading to $fdim(BS\Gamma(R)) \leq q$. Furthermore, based on Proposition \ref{ET} and Equation \ref{a}, it is clear that the FTMD of $\ BS(\Gamma(R) )$ is at least \( q \) when $q=2p-3$. Therefore, we can conclude that \( E \) is a minimum FTRS, i.e., \( fdim(BS\Gamma(R)) = q  \). Thus for the case $q=2p-3$, $fdim(BS\Gamma(R))= q$. \\ 
\end{proof}
\section {Conclusion}
\noindent Examining the FTMD of a graph is essential for ensuring effective communication and control in systems that may face node failures. In this paper, we conduct an in-depth study of the FTMD of the barycentric subdivision of zero divisor graph denoted by \( fdim(BS(\Gamma(\mathbb{Z}_{pq}))) \), where \( p \) and \( q \) are two distinct odd prime numbers with the condition that \( q \) is greater than \( p \). We establish that when \( q \) $>$ \( 2p - 1 \), then \( fdim(BS(\Gamma(\mathbb{Z}_{pq}))) = q - 1 \). This finding is significant, as it provides a clear metric indicating how node failures can impact the overall functioning of the graph. Additionally, in the scenario where \( p \) is greater than or equal to 3 and \( q \) equals \( 2p - 1 \) $\&$ $2p - 3$, we have demonstrated that \( fdim(BS(\Gamma(\mathbb{Z}_{pq}))) = q \). This result underscores the relationship between the primes \( p \) and \( q \) and their implications for fault tolerance within the graph. Investigating the FTMD of the barycentric subdivision of $\Gamma(\mathbb{Z}_{pq})$, where $p$ and $q$ are two distinct odd primes and $p< q \leq 2p-5$ will be an exciting topic for future research. 
% It may provide beneficial information on the stability of complex networks under various structural failure scenarios.
\section*{Authors’ Contributions}
Vidya: conceptualization, methodology, investigation, writing–original draft, writing–review
and editing. Sunny Kumar Sharma: Supervision, conceptualization, methodology, and validation. Prasanna Poojary: Supervision, conceptualization, validation, writing–review and editing. Vadiraja Bhatta G R: Supervision, writing–review and editing.

\section*{Use of Generative-AI tools declaration}
The authors declare that they have not used Artificial Intelligence (AI) tools in the creation of this
article.
\section*{Data Availability}
Data sharing does not apply to this manuscript because no data sets were analysed or generated during this particular study.
\section*{Conflicts of Interest}
Authors have nothing to declare as a conflict of interest.

% \section*{Acknowledgements}
% All the authors are thankful to their institute for providing a learning environment and continuous encouragement.
  				
	% \bibliographystyle{ieeetr}
	% \bibliography{name.bib}
   
  \end{document}